\documentclass[12pt]{amsart}

\usepackage{amsmath,amssymb,color,mathtools,listings,amsfonts,amsthm,stmaryrd,rotating,stmaryrd}
\usepackage{hyperref}

\usepackage{tikz,tkz-tab, tikz-cd}

\usepackage{soul}

\newtheorem{thm}{Theorem}

\newtheorem{cor}[thm]{Corollary}

\newtheorem{lem}[thm]{Lemma}
\newtheorem{prop}[thm]{Proposition}
\newtheorem{defn}[thm]{Definition}
\newtheorem{rmk}[thm]{Remark}

\newcommand{\id}{\mathrm{id}}

\newcommand{\NN}{{\mathbb N}}
\newcommand{\un}{{\mathbf 1}}

\def\scal#1{{\langle#1\rangle}}
\def\zero#1{{\langle#1\rangle}}

\title[Algebraic groups in non-commutative probability]{Algebraic groups in non-commutative\\ probability theory revisited\\[0.2cm]
}

\author[I.~Chevyrev]{Ilya Chevyrev}
\address{School of Mathematics, The University of Edinburgh, and The Maxwell
Institute for the Mathematical Sciences, James Clerk Maxwell Building, The King's Buildings,
Edinburgh, EH9 3FD, United Kingdom}
\email{ichevyrev@gmail.com}
\urladdr{https://ilyachevyrev.wordpress.com/}

\author[K.~Ebrahimi-Fard]{Kurusch~Ebrahimi-Fard}
\address{Department of Mathematical Sciences, Norwegian University of Science and Technology, NO 7491 Trondheim, Norway.}
\email{kurusch.ebrahimi-fard@ntnu.no}
\urladdr{https://folk.ntnu.no/kurusche/}

\author[F.~Patras]{Fr\'ed\'eric~Patras}
\address{Univ.~C\^ote d'Azur, CNRS, UMR 7351, Parc Valrose, 06108 Nice Cedex 02, France.}
\email{patras@unice.fr}
\urladdr{www-math.unice.fr/$\sim$patras}

\begin{document}

\begin{abstract}
The role of coalgebras as well as algebraic groups in non-commutative probability has long been advocated by the school of von Waldenfels and Sch\"urmann. Another algebraic approach was introduced more recently, based on shuffle and pre-Lie calculus, and results in another construction of groups of characters encoding the behaviour of states. Comparing the two, the first approach, recast recently in a general categorical language by Manzel and Sch\"urmann, can be seen as largely driven by the theory of universal products, whereas the second construction builds on Hopf algebras and a suitable algebraization of the combinatorics of noncrossing set partitions. Although both address the same phenomena, moving between the two viewpoints is not obvious. We present here an attempt to unify the two approaches by making explicit the Hopf algebraic connections between them. Our presentation, although relying largely on classical ideas as well as results closely related to Manzel and Sch\"urmann's aforementioned work, is nevertheless original on several points and fills a gap in the non-commutative probability literature. In particular, we systematically use the language and techniques of algebraic groups together with shuffle group techniques to prove that two notions of algebraic groups naturally associated with free, respectively Boolean and monotone, probability theories identify. We also obtain explicit formulas for various Hopf algebraic structures and detail arguments that had been left implicit in the literature.\\

\noindent{\footnotesize{MSC classes: 16T05, 16T10, 16T30, 17A30, 46L53, 46L54}}
\end{abstract}

\maketitle


\tableofcontents



\section{Introduction}

The school of von Waldenfels and Sch\"urmann has advocated the use of coalgebras as well as algebraic groups and their Hopf algebras of coordinate functions in the context of non-commutative probability theory \cite{bengohrschur,bengohrschur_02,bengohrschur_05,schuermann93,vWaldenfels85}. More recently, two of the present authors introduced a different approach, which is mainly driven by shuffle and pre-Lie calculus, leading to another construction of groups of characters \cite{ebrahimipatras_15,ebrahimipatras_16,ebrahimipatras_17,ebrahimipatras_19,ebrahimipatras_20}. The first construction can be understood as largely determined by the theory of universal products and additive free, monotone and boolean convolutions, whereas the second builds on a suitable algebraization of the combinatorics of lattices of set partitions and is encoded into a Hopf algebra structure on $T(T^+(\mathcal A))$, the double tensor algebra over an algebra $\mathcal A$ of non-commutative random variables. Although both approaches address the same phenomena, moving from one to the other has no obvious mathematical formulation.

The goal of the work at hand is to present a unified Hopf algebraic point of view that permits to compare the different approaches and to make the connections transparent by including explicit formulas on the Hopf algebraic structures. The key object appears to be $S(T^+(\mathcal A))$, the Hopf algebra of polynomials over the tensor algebra over $\mathcal A$. Our exposition of results, while relying on classical ideas, is original on several points. Moreover, we obtain explicit formulas for the Hopf algebraic structures and detail arguments that have been sometimes left implicit in the literature.

Although we have tried to give a self-contained presentation, we need to assume some familiarity with the foundations of non-commutative probability theory as well as the theory of Hopf algebras. The reader who would like to see more details on non-commutative probability and related topics is referred to \cite{biane_02,mingospeicher_17,nicaspeicher_06,voiculescu_92,voiculescu_95}. Detailed studies of the notion of universal product in non-commutative probability can be found in  \cite{bengohrschur_02,manzelschuermann_16,muraki_03,speicher_97b}. Regarding background and details on Hopf and pre-Lie algebras, in particular of combinatorial nature, we refer the reader to the book \cite{cp2021}. 

\medskip

The paper is organised as follows. The next section introduces a collection of notations and conventions used throughout the paper. In Section \ref{sect:group}, we consider the group structure coming from convolution in the case of classical probability. Section \ref{sect:up} presents the notion of universal product in the context of non-commutative probability. Section \ref{sect:GS} recalls the shuffle algebra approach to moment-cumulant relations in non-commutative probability. Sections \ref{sec:monotonestate},  \ref{sect:bg} and \ref{sect:fg}, discuss the monotone, free and Boolean group structures for states.

\medskip

\noindent {\bf{Acknowledgements}}: The authors would like to thank the anonymous referee for their detailed comments that helped to improve the presentation of the paper.

This work was partially supported by the project “Pure Mathematics in Norway”, funded by Trond Mohn Foundation and Troms{\o} Research Foundation.
IC acknowledges support by the EPSRC via the New Investigator Award EP/X015688/1.
KEF was supported by the Research Council of Norway through project 302831 “Computational
Dynamics and Stochastics on Manifolds” (CODYSMA). FP acknowledges support from the European Research Council (ERC) under the European Union’s Horizon 2020 research and innovation program (Duall project, grant agreement No. 670624), from the ANR project Algebraic Combinatorics, Renormalization, Free probability and Operads -- CARPLO (Project-ANR-20-CE40-0007) and from the ANR--FWF project PAGCAP (ANR-21-CE48-0020, Austrian Science Fund (FWF), grant I 5788).


\section{General notation and conventions}
\label{sec:notation}
 Let $\NN^\ast$ denote the positive integers. For any $n \in \NN^\ast$, we define the sets $[n] := \{1,\ldots, n\}$ and $\zero{n} := \{0, \ldots, n-1\}$. Throughout the article, $(\mathcal{A},\varphi)$ stands for a non-commutative probability space. This means that $\mathcal{A}$ is an associative algebra with unit $1_\mathcal{A}$ over a field $\mathbb{K}$ of characteristic zero (usually taken to be the complex numbers in concrete applications). See the standard reference \cite{nicaspeicher_06} for more details.  It is important to distinguish the unit of $\mathcal{A}$ from the unit of the ground field: we denote the former $1_\mathcal{A}$ (or $1$) and the latter $\mathbf 1$. The map $\varphi$ is an unital $\mathbb{K}$-valued linear form on $\mathcal{A}$, i.e., $\varphi(1_\mathcal{A}) = \mathbf 1$. The elements of $\mathcal{A}$ play the role of random variables, while the map $\varphi$ should be considered as the expectation. 

\medskip

{\bf{Classical cumulants}} are defined in probability where the algebra $\mathcal{A}$ of random variables is commutative, as a family of multilinear functionals $\{c_n: \mathcal{A}^n \to \mathbb{C}\}_{n\geq1}$ iteratively determined by classical moment-cumulant relations
\begin{equation}
\label{eq:clMCrel}	
	\varphi(a_1\cdots a_n) = \sum_{\pi\in P_n} \prod_{\pi_i \in \pi} c_{|\pi_i|}(a_{1},\ldots,a_{n}| \pi_i),
\end{equation}
see~\cite[Sec.~3.2]{peccati_taqqu_11}.
The sum on the right hand side runs over all set partitions of $\{1,\ldots,n\}$. Here and below we write for a block $\pi_i=\{i_1,\ldots,i_{|\pi_i|}\}$ in a set partitions $\pi=\{\pi_1,\ldots,\pi_m\} \in P_n$ and a multilinear functional $f$
$$
	f_{|\pi_i|}(a_{1},\ldots,a_{n}| \pi_i) := f_{|\pi_i|}(a_{i_1},\ldots,a_{i_{|\pi_i|}}).
$$ 
\linebreak
\indent {\bf{Free, Boolean, and monotone cumulants}} in non-commutative probability provide effective combinatorial tools for studying non-commutative distributions of random variables. Free cumulants \cite{speicher1994multiplicative,nicaspeicher_06} form a family of multilinear functionals $\{\text{k}_n: \mathcal{A}^n \to \mathbb{C}\}_{n\geq1}$ which are recursively defined by free moment-cumulant relations
\begin{equation}
\label{eq:freeMCrel}	
	\varphi(a_1\cdots a_n) = \sum_{\pi\in NC_n} \prod_{\pi_i \in \pi} \text{k}_{|\pi_i|}(a_{1},\ldots,a_{n}| \pi_i).
\end{equation}
Boolean cumulants \cite{speicher_97a} form another family of multilinear functionals $\{ \text{b}_n:\mathcal{A}^n\to\mathbb{C}\}_{n\geq1}$ recursively defined by Boolean moment-cumulant relations
\begin{equation}
\label{eq:booleanMCrel}
	\varphi(a_1\cdots a_n)= \sum_{\pi\in \operatorname{Int}_n}  \prod_{\pi_i \in \pi}  \text{b}_{|\pi_i|}(a_{1},\ldots,a_{n}| \pi_i).
\end{equation}
Monotone cumulants \cite{hasebesaigo_11} form a family of multilinear functionals $\{h_n:\mathcal{A}^n\to\mathbb{C}\}_{n\geq1}$ recursively defined by the monotone moment-cumulant relations
\begin{equation}
\label{eq:monotoneMCrel}
	\varphi(a_1\cdots a_n)= \sum_{\pi\in NC_n} \frac{1}{t(\pi)!} \prod_{\pi_i \in \pi} h_{|\pi_i|}(a_{1},\ldots,a_{n}| \pi_i).
\end{equation}
The sums on the right hand sides of \eqref{eq:freeMCrel} and \eqref{eq:monotoneMCrel} run over the lattice $NC_n$ of non-crossing partitions of the set $\{1,\ldots,n\}$. The sum on the right hand side of \eqref{eq:booleanMCrel} on the other hand is defined over the lattice $\operatorname{Int}_n \subset NC_n$ of interval partitions (i.e., where each block is an interval).

\medskip

{\bf{Tensor algebras}}:
The non-unital tensor algebra $\bigoplus_{n > 0} \mathcal{A}^{\otimes n}$ over $\mathcal{A}$ is denoted $T^+(\mathcal{A})$. On the other hand, the unital tensor algebra is obtained by adding a copy of the ground field, that is, $T(\mathcal{A}):=\mathbb{K} \oplus T^+(\mathcal{A})$. The same conventions apply for the tensor (resp.~non-unital tensor) algebra $T(V)$ (resp.~$T^+(V)$) over a vector space $V$.

In general, when $B=\bigoplus_{n\in\NN}B_n$ is a graded connected algebra (meaning that the degree zero component of the graded algebra is isomorphic to the ground field), we will write $B^+$ for its augmentation ideal, that is, $\bigoplus_{n\in\NN^\ast}B_n$.

Tensors $a_1 \otimes \dots \otimes a_n \in T(\mathcal{A})$ are written as words $a_1 \cdots a_n$. They have to be distinguished from the product of the $a_i$ seen as elements in the algebra $\mathcal{A}$, which is written $a_1\cdot_{\!\!\scriptscriptstyle{\mathcal{A}}} \cdots  \cdot_{\!\scriptscriptstyle{\mathcal{A}}} a_n$. The unit $\mathbf 1$ of the ground field is identified with the empty word $\emptyset$. The word $a \cdots a \in \mathcal{A}^{\otimes n}$ is denoted $a^{ n}$ -- and should not be confused with the $n$-th power of $a$ in $\mathcal{A}$, which is written $a^{\cdot_{\!\!\scriptscriptstyle{\mathcal{A}}} n}$. 

The natural algebra structure on $T^+(\mathcal{A})$ refers to the concatenation product of words, which turns it into a non-commutative graded algebra with the natural degree of a word $w=a_1 \cdots a_n$ being its length, i.e., its number of letters, $\deg(w):=n$. In the unital case the empty word is of length zero and, as we saw already, plays the role of the unit for the concatenation product.

We use here the convention that identifies the algebra of polynomials $S(V)$ (also known as the symmetric tensor algebra) over a vector space $V$ with the space of covariants $\bigoplus_{n\geq 0}T^n(V)/{S_n}$, where $S_n$ is the $n$-th symmetric group acting on tensors of length $n$ by permutation:
$$
	\forall\sigma\in S_n,\ v_1,\dots,v_n\in V,\quad 
	\sigma(v_1\cdots v_n):=v_{\sigma^{-1}(1)}\cdots v_{\sigma^{-1}(n)}.
$$ 
The augmentation ideal $S^+(V)$ is $\bigoplus_{n> 0}T^n(V)/{S_n}$. The same approach is used for the algebra of polynomials $S(X)$ over a set $X$.

We define then $S(T^+(\mathcal{A}))$ and $S^+(T^+(\mathcal{A}))$ to be the polynomial algebra, respectively the non-unital polynomial algebra over $T^+(\mathcal{A})$. We again identify $\mathbf 1\in \mathbb{K}$ with the unit of $S(T^+(\mathcal{A}))$ to distinguish it from the unit of $\mathcal A$ and thus treat $\mathbb{K}$ as a subspace of $S(T^+(\mathcal{A}))$.
Notice that, as the ground field is of characteristic zero, there is a canonical bijection between invariants and covariants, and $S(T^+(\mathcal{A}))$ identifies with the space of symmetric tensors over $T^+(\mathcal{A})$. See \cite[Sect.~2.4]{cp2021} for a general discussion of polynomial (Hopf) algebras.

Elements in $S(T^+(\mathcal{A}))$ are represented using a parenthesized bar-notation, i.e., $(w_1 | \cdots | w_n) \in S(T^+(\mathcal{A}))$, where the words $w_i \in T^+(\mathcal{A})$, $i=1,\ldots,n$. The algebra $S(T^+(\mathcal{A}))$ is commutative with the bar-product $a|b := (w_1 | \cdots | w_n | v_1 | \cdots | v_m)$ for $a= (w_1 | \cdots | w_n)$ and $b= ( v_1 | \cdots | v_m)$. We denote by $\deg(a):=\deg(w_1)+\dots+\deg(w_n)$ the degree of an element $a\in S(T^+(\mathcal{A}))$.

Analogous conventions hold for both the non-unital and unital double tensor algebras, defined to be $T^+(T^+(\mathcal{A})):=\bigoplus_{n > 0} T^+(\mathcal{A})^{\otimes n}$, respectively $T(T^+(\mathcal{A})):=\bigoplus_{n \ge 0} T^+(\mathcal{A})^{\otimes n}$. Elements are denoted using the bracket-bar notation, $[w_1 | \cdots | w_n] \in T^+(T^+(\mathcal A))$, for words $w_i \in T^+(\mathcal{A})$, $i=1,\ldots,n$.
\medskip

{\bf{Set theoretic conventions}}:
Consider now a canonically ordered subset $S:=\{s_1, \ldots , s_m\}$ of the set  $[n]$. For a word $a_1 \cdots a_n \in T^+(\mathcal{A})$, we set $a_S := a_{s_1} \cdots a_{s_m} \in T^+(\mathcal{A})$ and $a_\emptyset:=\mathbf{1}$. A connected component $S':=\{s_i,\ldots,s_{i+j}\} \subseteq S \subseteq [n]$ is a maximal sequence of successive elements in $S$, i.e.,
$s_{i+k}=s_i+k$ for all $1\leq k\leq j$, and $s_{i-1}<s_i-1$ and $s_{i+j+1}>s_{i+j}+1$ whenever $i>1$ and $i+j<m$ respectively.
For any $S \subseteq [n]$ we denote by $S_1, \ldots, S_p$ and $J^S_1, \ldots, J^S_k$ the connected components of $S$ and of $[n] - S$, respectively. Both are naturally ordered in terms of their minimal (or maximal) elements. For example, if $n=7$ and $S=\{2,3,6\}$, we get $p=2$ connected components of $S$, $S_1=\{2,3\}$, $S_2=\{6\}$ and $k=3$ connected components of $[n]-S$, $J^S_1=\{1\}$, $J^S_2=\{4,5\}$, $J^S_3=\{7\}$.

For two partitions $\pi_1,\pi_2$ of $[n]$, we say that $\pi_1$ is finer than $\pi_2$ and write $\pi_1\leq \pi_2$ if every block of $\pi_1$ is a subset of one block in $\pi_2$.
The relation $\leq$ turns the set of partitions into a lattice.
With this notation, $\pi_{\operatorname{Int}}(S):=\{S_1,\ldots,S_p,J^S_1,\ldots, J^S_k\}$ is the maximal interval partition which is finer than the two-block partition $\{S,[n]-S\}$.
We note that this is a dual closure map in the sense of~\cite[1.4~Closures]{blyth2005}.


\section{Group structures from convolutions: the classical case}
\label{sect:group}

Let us start with classical probability as a template for the non-commutative constructions which are to come in the next section. The following presentation revisits ideas from classical probability in an algebraic language recasting them in the setting presented in the recent work \cite{efptz}, which formulates moment-cumulant relations as well as Wick polynomials in the framework of algebraic groups.

We assume that all distributions of random variables have determinate moment problem, that is, they are entirely characterized by the sequence of all their moments (think for example to random variables with compact support, Gaussian variables, or alike; see \cite{Akhiezer} as a standard reference on these questions). This hypothesis arises from a parallel assumption made within the combinatorial approach to free probability, where the distribution of a random variable is identified with the sequence of its moments at all orders. See for example the definition of distributions in free probability in the standard reference \cite[Chapter 1]{nicaspeicher_06}.

Consider now a commutative unital algebra $\mathcal A$ of classical random variables with expectation value operator $\mathbb E$. It is convenient to encode distributions of random variables by the linear form still written $\mathbb E:S(\mathcal A)\to \mathbb K,$
$$ 
	a_1\cdots a_n\longmapsto 
	\mathbb E(a_1\cdots a_n)=
	\mathbb E[a_1\cdot_{\!\!\scriptscriptstyle{\mathcal{A}}} \cdots 
	\cdot_{\!\!\scriptscriptstyle{\mathcal{A}}}  a_n],
$$
where $\mathbb E(\mathbf 1):=1$ and $a_1\cdots a_n$ stands for a monomial of degree $n$ in the polynomial algebra $S(\mathcal A)$ (not to be confused with $a_1\cdot_{\!\!\scriptscriptstyle{\mathcal{A}}} \cdots \cdot_{\!\!\scriptscriptstyle{\mathcal{A}}}  a_n$, the product of the $a_i$ in $\mathcal A$). The distribution of an element $a\in \mathcal A$ is then obtained as the sequence of its moments $\mathbb E (a^n)$. The (additive) convolution of two distributions of scalar-valued random variables $X$, $Y$ belonging to $\mathcal A$  is then defined as the distribution of the sum $X'+Y'$, where $X'$ and $Y'$ are independent and have the same laws as $X$ and $Y$. The simplest way to make this process effective is to construct the tensor product algebra $\mathcal A\otimes\mathcal A$ equipped with the expectation value operator $\mathbb E\otimes\mathbb E$. Given two integrable functions $f(X)$ and $g(Y)$, of $X$ respectively  $Y$, one gets
$$
	\mathbb E\otimes\mathbb E(f(X)\otimes g(Y))=\mathbb E(f(X))\mathbb E(g(Y)),
$$
which shows that $X':=X\otimes 1$ and $Y':=1\otimes Y$ are two independent copies in $\mathcal A\otimes\mathcal A$ of the random variables $X$ and $Y$. 

This process dualizes as follows.
\begin{defn} 
Let $\phi$ and $\psi$ be two unital linear forms on $S(\mathcal A)$. The (additive) convolution product of $\phi$ and $\psi$ is the unital linear form on $S(\mathcal A)$ defined by:
$$
	\phi\ast\psi(\mathbf 1):=1,
$$
\begin{equation*}
	\phi\ast\psi(a_1\cdots a_n)
	:=\phi\otimes\psi((a_1\otimes \mathbf 1+\mathbf 1\otimes a_1)\cdots (a_n\otimes \mathbf 1+\mathbf 1\otimes a_n)).
\end{equation*}
\end{defn}
\begin{cor}
If $X$ and $X'$ are two independent identically distributed (i.i.d.) random variables, then the distribution of $X+X'$ with respect to $\mathbb E$ identifies with the distribution of $X$ with respect to  $\mathbb E\ast \mathbb E$.\end{cor}

Furthermore, as $\phi\ast\psi(a_1\cdots a_n)=\phi(a_1\cdots a_n)+\psi(a_1\cdots a_n)+ \text{l.o.t.},$ where l.o.t.~refers to a sum of products of evaluations of $\phi$ and $\psi$ on tensors of length strictly less than $n$, a standard inductive argument (on the length of tensors) that we omit shows that the $\ast$-product defines a group structure on the set of unital linear forms on $S(\mathcal A)$. 

\begin{defn} We call the set $\mathcal G$  of unital linear forms on $S(\mathcal A)$, equipped with the convolution product, the group of generalized measures on $\mathcal A$.
\end{defn}

Let us assume now that $\mathcal A$ has a basis $\mathcal B=(b_i)_{i\in\NN}$. Think for example of a Gaussian family $\{\Omega_1,\dots,\Omega_n\}$ and the algebra spanned by monomials in the $\Omega_1,\dots,\Omega_n$ (which is a dense subalgebra of the square integrable functions).

Given a sequence $I=(i_1,\dots,i_k)$ with $0\leq i_1\leq\dots \leq i_k$, we set $b_I:=b_{i_1}\cdots b_{i_k}\in S(\mathcal A)$. As usual, $b_I$ should \it not \rm be confused with the product $b_{i_1}\cdot_{\!\!\scriptscriptstyle{\mathcal{A}}}\cdots \cdot_{\!\!\scriptscriptstyle{\mathcal{A}}} b_{i_k}$ in $\mathcal A$. The set $\mathcal X$ of all monomials $b_I$ defines then a (complete) system of coordinates on $\mathcal G$, where, for an element $\phi\in\mathcal G$, we specify
$$
	\langle\phi |b_I\rangle:= \phi(b_I).
$$
Here complete refers to the fact that this set $\mathcal X$ of coordinate functions induces a bijective correspondence between elements of $\mathcal G$ and families of scalars indexed by sequences of integers: 
$$
	\phi\longmapsto \{\phi(b_I)\}_{I}.
$$
The definition of the group law on $\mathcal G$ implies
$$
	\langle\phi\ast\psi |b_I\rangle=\sum\limits_{A\coprod B = [|I|]} 
	\langle\phi|b_{A,I}\rangle\langle\psi|b_{B,I}\rangle,
$$
where, if $A=\{j_1,\ldots ,j_k\}$, $b_{A,I}:=b_{i_{j_1}} \cdots b_{i_{j_k}}$. This shows that the $b_I$ are representative functions on the group $\mathcal G$. See \cite[2.7 Representative functions]{cp2021} for details. By standard arguments, the polynomial algebra $S(\mathcal X)$ is a Hopf algebra whose coproduct $\delta$ is given on generators by 
\begin{equation}
\label{copro1}
	\delta(b_I)=\sum\limits_{A\coprod B= [|I|]} b_{A,I}\otimes b_{B,I},
\end{equation}
and the following theorem holds (see \cite[3.5 Algebraic groups]{cp2021}).

\begin{thm}
The (commutative) group $\mathcal G$ of generalized measures is, up to a canonical isomorphism, the group of characters (algebra maps to the ground field) of the (commutative and cocommutative) Hopf algebra $S(\mathcal X)$, with coproduct given by \eqref{copro1}. 
\end{thm}

\begin{rmk} 
A terminological precision: it follows from general arguments (namely the fundamental theorem for coalgebras), but is easy to check directly here, that the Hopf algebra $S(\mathcal X)$ is a direct limit of commutative Hopf subalgebras which are finitely generated as algebras. The group $\mathcal G$ is therefore an inverse limit of affine algebraic groups (groups of characters on a finitely generated commutative Hopf algebra), that is, a pro-affine algebraic group.
\end{rmk}

Let us show briefly why these constructions are meaningful for probability. Assume again that we consider a given expectation value functional $\mathbb E$ on $\mathcal A$ and view it as an element of $\mathcal G$. Let us write $\star$ for the convolution product of linear forms on $S(\mathcal X)$: given $\lambda,\beta\in S(\mathcal X)^\ast$,
$$
	\lambda\star\beta(b_{I_1}\cdots b_{I_k}):=(\lambda\otimes\beta)\circ\delta (b_{I_1}\cdots b_{I_k}).
$$
The product map $\ast$ on $\mathcal G$ is the restriction of $\star$ to the set of characters -- so, the two products agree on $\mathcal G$, but we prefer to keep a notational distinction between them for the sake of clarity.

\begin{prop}
\label{cumformu}
For elements $a_1,\dots ,a_n \in \mathcal A$, we have
$$ 
	\log^\star(\mathbb E)(a_1\cdots a_n)= c_n(a_1,\ldots ,a_n),
$$
where $c_n$ stands for the multivariate classical cumulant map as in~\eqref{eq:clMCrel}.
\end{prop}

\begin{proof}
Iterating the coproduct formula \eqref{copro1} gives
\begin{align*}	 
	\log^\star(\mathbb E)(a_1\cdots a_n)
	&=\sum\limits_{k=1}^n\frac{(-1)^{k-1}}{k}\sum\limits_{P \in OP^k_n}\mathbb E[a_{P_1}]\cdots \mathbb  E[a_{P_k}]\\
	&=\sum\limits_{k=1}^n(-1)^{k-1}(k-1)! \sum\limits_{P \in P^k_n}\mathbb E[a_{P_1}]\cdots \mathbb E[a_{P_k}],
\end{align*}
where $OP^k_n$ and $P^k_n$ stand for the set of ordered partitions of length $k$ respectively the set of partitions of length $k$ of $[n]$. The last term provides the Leonov--Shiryaev definition of multivariate cumulants, see~\cite[Sec.~3.2]{peccati_taqqu_11}. The combinatorial study of the moments/cumulants relations actually goes back at least to Sch\"utzenberger's thesis \cite{schutz}; see for instance  \cite[Eq.~(12)]{efptz} for details on an Hopf algebraic account).
\end{proof}

However, these ideas can be recast differently. Consider a basis element $b_I=b_{i_1}\cdots b_{i_k}$ in $\mathcal X$. We use Proposition \ref{cumformu} to introduce another system of coordinates on $\mathcal G$. Let $\mathcal Y$ be the set of the $c_I$, where, setting $n:=|I|$,
$$
	c_I:=\sum\limits_{k=1}^n(-1)^{k-1}(k-1)!\sum\limits_{P\in P^k_n}b_{P_1,I}\cdots b_{P_k,I}.
$$
As $c_I=b_I+\text{l.o.t.}$, the $c_I$ indeed define another complete system of coordinates on $\mathcal G$. Furthermore, the same argument as for $\mathbb E$ gives for $\alpha \in \mathcal G$ (viewed as a character on $S(\mathcal X)$),
$$
	\log^\star(\alpha)(b_I)=\langle\alpha|c_I\rangle.
$$

As the products $\ast$ and $\star$ are commutative and both agree on $\mathcal G$, we obtain for given $\alpha, \beta \in \mathcal G$:
\begin{align*}	
	\langle\alpha\ast\beta|c_I\rangle
	&=\log^\star(\alpha\ast\beta)(b_I)\\
	&=\log^\star(\alpha)(b_I)+\log^\star(\beta)(b_I)\\
	&=\langle\alpha|c_I\rangle+\langle\beta|c_I\rangle.
\end{align*}
This formula can be interpreted as follows.
As the $\star$ product is dual to the coproduct on  $S(\mathcal X)$, the formula expresses the fact that the $c_I$ are primitive elements in $S(\mathcal X)$.
It follows that, in the system of coordinates $\mathcal Y$, the group $\mathcal G$ obeys the additive group law: the $c_I$-coordinate of $\alpha\ast\beta$ is the sum of the $c_I$-coordinates of $\alpha$ and $\beta$.

Several of the phenomena we have accounted for in the classical commutative case will resurface in the non-commutative context, although some differences will emerge as well, particularly when examining the so-called monotone and anti-monotone cases, which we will introduce in the upcoming section.


\section{Universal products: the non-commutative case}
\label{sect:up}

Consider now a non-commutative probability space  $(\mathcal A,\varphi)$. The distributions of random variables (seen as elements of $\mathcal A$) are then encoded by the linear form (see again \cite{nicaspeicher_06} for details), $\phi: T(\mathcal A)\to \mathbb K,$
$$ 
	a_1\cdots a_n\longmapsto \mathbb 
	\phi(a_1\cdots a_n)=\varphi(a_1\cdot_{\!\!\scriptscriptstyle{\mathcal{A}}}  \cdots \cdot_{\!\!\scriptscriptstyle{\mathcal{A}}}  a_n),
$$
where $a_1\cdots a_n$ stands now for a non-commutative monomial of degree $n$ in the tensor algebra over $\mathcal A$ (once again, not to be confused with the product of the $a_i$ in $\mathcal A$). The distribution of $a\in \mathcal A$ is then obtained as the sequence of its moments $\phi (a^n)$. 

Analogously to the commutative case, the (additive) convolution of two distributions of random variables $X$, $Y$ belonging to $\mathcal A$  could then be defined as the distribution of the sum $X'+Y'$, where $X'$ and $Y'$ are independent, having the same laws as $X$ and $Y$. However, the analogy stops here. Indeed, whereas in the commutative case there is a unique notion of independence, there exist five distinct notions of independence in the non-commutative setting assuming certain axioms to hold: see \cite{muraki_03} and the following developments in this section. The best way to approach this phenomenon in the context of the present article is through the notion of universal product -- devised, among others, to create independent copies of random variables in the various non-commutative probability theories. We shall recall it below, and briefly explain how it gives rise to five distinct group structures that will be studied later on in the article. Note that we omit the discussion of the so-called tensor case which relies on constructions similar to the commutative case.

\smallskip

Let us write $Alg_n$ for the category of non-unital associative algebras (of course, an algebra in $Alg_n$ can have a unit, as $\mathcal A$ does, but its existence is not required for the constructions presented here). Recall first the construction of free products.

\begin{lem}
The coproduct or free product\index{Free product} of algebras, denoted by the symbol $\oast$, in the category $Alg_n$ is obtained as follows: let ${\mathcal A}_1,{\mathcal A}_2$ be two associative algebras, then:
$$
	{\mathcal A}_1\oast {\mathcal A}_2
	:=\bigoplus\limits_{\varepsilon\in Alt}{\mathcal A}_{\varepsilon_1}\otimes 
	\cdots \otimes {\mathcal A}_{\varepsilon_n},
$$
where $Alt$ is the set of sequences $(\varepsilon_1,\dots,\varepsilon_n)\in\{1,2\}^n$ such that $\varepsilon_{i+1}\not=\varepsilon_i$ for $i<n$. The canonical embeddings of ${\mathcal A}_1$ and ${\mathcal A}_2$ in ${\mathcal A}_1\oast {\mathcal A}_2$ are denoted respectively $\iota_1$ and $\iota_2$. An element in a component ${\mathcal A}_{\varepsilon_1}\otimes \dots \otimes {\mathcal A}_{\varepsilon_n}$ will be called an {\it alternating tensor}.
\end{lem}

The product of two tensors $h_1\otimes \cdots \otimes h_n$ and $h_1'\otimes \cdots \otimes h_m'$ in ${\mathcal A}_1 \oast {\mathcal A}_2$ is defined as the concatenation product $h_1\otimes  \cdots \otimes h_n\otimes h_1'\otimes  \cdots \otimes h_m'$, if $h_n$ and $h_1'$ respectively belong to $ {\mathcal A}_1$ and $ {\mathcal A}_2$ or $ {\mathcal A}_2$ and $ {\mathcal A}_1$ and otherwise as:
$h_1\otimes \cdots \otimes (h_n \cdot_{\scriptscriptstyle{\mathcal A}_i} h_1')\otimes \cdots \otimes h_m'$, where $h_n,h_1'\in {\mathcal A}_i$, $i=1,2$. We will use from now on the word notation for tensor products in ${\mathcal A}_1\oast {\mathcal A}_2$: $h_1 \cdots h_n:=h_1\otimes \cdots \otimes h_n$.

The problem addressed by the theory of universal products is the one of extending linear forms $\phi_1$ and $\phi_2$ defined respectively on ${\mathcal A}_1$ and on ${\mathcal A}_2$ to a linear form $\phi_1\phi_2$ on the free product ${\mathcal A}_1 \oast {\mathcal A}_2$. One requires the construction to be associative, natural (that is, functorial in ${\mathcal A}_1$ and $ {\mathcal A}_2$)  and to satisfy the normalization conditions 
\begin{align*}
	(\phi_1\phi_2)\circ \iota_1 
	&=\phi_1\\
 	(\phi_1\phi_2)\circ \iota_2
	&=\phi_2\\
	(\phi_1\phi_2)(\iota_1(a)\iota_2(b))
	&=(\phi_1\phi_2)(\iota_2(b)\iota_1(a))=\phi_1(a)\phi_2(b),
\end{align*}
for $a\in {\mathcal A}_1$ and $ b\in {\mathcal A}_2$. These requirements are respectively axioms $U_2,U_3,U_4$ of Muraki \cite{muraki_03}.

Solutions were classified by Speicher, Sch\"urmann, and Muraki. They take the following five forms \cite[Thm.~2.2]{muraki_03}. In the subsequent definitions, $\phi_1$ and $\phi_2$ are linear forms on ${\mathcal A}_1$ respectively ${\mathcal A}_2$. 

\begin{defn}[Universal Boolean product] 
For an alternating tensor $a_1\cdots a_n$,
$$ 
	\phi_1\bullet\phi_2(a_1\cdots a_n)
	:=\Big(\prod_{a_i\in {\mathcal A}_1}\phi_1(a_i)\Big)
	\Big(\prod_{a_j\in {\mathcal A}_2}\phi_2(a_j)\Big).
$$
\end{defn}

\begin{defn}[Universal free product]
\label{defff} 
For an alternating tensor $a_1\cdots a_n$,
\begin{equation}
\label{eq:defff} 
	\phi_1\ast_f\phi_2(a_1\cdots a_n)
	:=\sum\limits_{I\subsetneq [n]}(-1)^{n-|I|+1}
	\phi_1\ast_f\phi_2\big(\prod_{i\in I}a_i\big)\Big(\prod_{\substack{j\notin I \\ a_j\in {\mathcal A}_1}}\phi_1(a_j)\Big)\Big(\prod_{\substack{k\notin I \\ a_k\in {\mathcal A}_2}}\phi_2(a_k)\Big).
\end{equation}
\end{defn}

In this definition (by induction), it is understood that $\phi_1\ast_f\phi_2(\prod_{i\in \emptyset}a_i):=1$ and that products $\prod_{i\in I}a_i$ are taken in increasing order of the indices (e.g., $\prod_{i\in \{3,5,7\}}a_i=a_3a_5a_7$). This convention for handling products will remain consistent throughout the entire article.

\begin{rmk}\label{closefo}
It is difficult to get an explicit closed formula for the universal free product \eqref{eq:defff}, and we are not aware of such a formula in the literature. However, the following can be said. Recall first that a noncrossing partition of a totally ordered set can always be obtained (although not uniquely) by recursively extracting intervals from the totally ordered set. For example, the noncrossing partition $\{1,7\}\coprod \{2,4,6\}\coprod \{3\}\coprod\{5\}$ can be obtained by first extracting $\{3\}$ (the remaining ordered set is $\{1,2,4,5,6,7\}$), then $\{5\}$  (the remaining ordered set is $\{1,2,4,6,7\}$), then $\{2,4,6\}$ (which is an interval in $\{1,2,4,6,7\}$) to be left with the last block, $\{1,7\}$. Conversely, any partition obtained by the process of recursively extracting intervals is noncrossing. 

Definition \ref{defff} then implies (by an elementary inductive argument based on this characterisation of noncrossing partitions) that there exist universal coefficients $\alpha_\pi$ such that 
$$
	\phi_1\ast_f\phi_2(a_1\cdots a_n)
	=\sum\limits_{\substack{\pi=(\pi_1,\cdots,\pi_k)\in NC_n\\ \pi \leq \{S,S^c\}}}\alpha_\pi\prod\limits_{i=1}^k
	\phi_{1,2}(\prod\limits_{j\in\pi_i}a_j),
$$
where $NC_n$ is the set of noncrossing partitions of $[n]$, $S=\{1,3,5,\ldots\}$ is the set of odd integers in $[n]$,
and where $\phi_{1,2}$ is $\phi_1$ or $\phi_2$ depending on whether the elements $a_j$ for $j\in\pi_i$ belong to $\mathcal A_1$ or $\mathcal A_2$.
Recall that $\pi \leq \{S,S^c\}$ means that $\pi$ is finer than $\{S,S^c\}$; remark that this condition implies that, if $j$ and $j'$ belong to the same block $\pi_i$, then $a_j$ and $a_{j'}$ belong to the same algebra $\mathcal A_1$ or $\mathcal A_2$.  
\end{rmk}

Let now $b_1\cdots b_n$ be a sequence of random variables belonging to $\mathcal A_1$ and $\mathcal A_2$. We do not require anymore the sequence to be alternating, that is: if $b_i\in \mathcal A_{1}$ (resp.~$\mathcal A_2$), then it is not necessarily the case that $b_{i+1}\in \mathcal A_2$ ({resp.}~$\mathcal A_{1}$). We denote by $S$ the set of indices $i$ such that $b_i \in \mathcal A_1$.
We say that a noncrossing partition $\pi \in NC_n$ of $[n]$ is {\it adapted to $S$} if and only if
\[
\pi_{\operatorname{Int}}(S) \leq \pi \leq \{S,[n]-S\},
\]
i.e., there exists a splitting $\pi=\pi^1\cup \pi^2$ of the partition into two subsets of blocks such that
\begin{itemize}
\item $\pi^1$ is a partition of $S$ (and therefore $\pi^2$ a partition of $[n]-S$),
\item given $i,\ 1\leq i<n$, in a block $\pi_j$ of $\pi^p$ (where $p\in\{1,2\}$), then either $i+1\in \pi_j$ or $i+1\in \pi_l$ with $\pi_l\in \pi^{q}$ with $q\not= p$. 
\end{itemize}
The set of noncrossing partitions of $[n]$ adapted to $S$ is denoted $ANC_n^S$ and if such a splitting exists for a given $\pi \in NC_n$, then it is unique and we write $(\pi^1, \pi^2) \in ANC_n^S$. In words, a noncrossing partition $\pi$ is adapted to $S$ if and only if 1) elements in each of its blocks correspond to lower indices of random variables in the same algebra $\mathcal A_1$ or $\mathcal A_2$, and 2) when $a_i$ and $a_{i+1}$ belong to the same algebra, $i$ and $i+1$ belong to the same block. As an example, we consider elements $b_1,b_2,b_3 \in \mathcal A_1$, $a_1,a_2,a_3,a_4 \in \mathcal A_2$, and the word $w=b_1b_2 a_1 a_2  a_3 b_3  a_4$. Then $S=\{1,2,6\}$ and the pairs $\pi_1=\{ \{1,2\},\{6\} \}$, $\pi_2=\{ \{3,4,5\},\{7\} \}$ (see Fig.~\ref{aNC1}) as well as $\sigma_1=\{ \{1,2,6\} \}$, $\sigma_2=\{ \{3,4,5\},\{7\} \}$ (see Fig.~\ref{aNC2}) , and $\tau_1=\{ \{1,2\},\{6\} \}$, $\tau_2=\{ \{3,4,5,7\} \}$ (see Fig.~\ref{aNC3})  would be noncrossing partitions adapted to $S$.

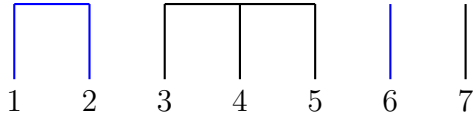
\begin{figure}
\begin{center}
\begin{tikzpicture}
\draw[thick,blue] (0,0) -- (1,0); 
\draw[thick,blue] (0,0) -- (0,-1) node[below,black]  {$1$}; \draw[thick,blue] (1,0) -- (1,-1) node[below,black]  {$2$}; 
\draw[thick,black] (2,0) -- (4,0); 
\draw[thick,black] (2,0) -- (2,-1) node[below,black]  {$3$}; \draw[thick,black] (3,0) -- (3,-1) node[below,black]  {$4$}; \draw[thick,black] (4,0) -- (4,-1) node[below,black]  {$5$}; 
\draw[thick,blue] (5,0) -- (5,-1) node[below,black]  {$6$}; 
\draw[thick,black] (6,0) -- (6,-1) node[below,black]  {$7$};  
\end{tikzpicture}
\caption{Adapted noncrossing partition $\pi_1 \cup \pi_2$.}\label{aNC1}
\end{center}
\end{figure}

\begin{figure}
\begin{center}
\begin{tikzpicture}
\draw[thick,blue] (0,0) -- (5,0); 
\draw[thick,blue] (0,0) -- (0,-1) node[below,black]  {$1$}; \draw[thick,blue] (1,0) -- (1,-1) node[below,black]  {$2$}; 
\draw[thick,black] (2,-0.5) -- (4,-0.5); 
\draw[thick,black] (2,-0.5) -- (2,-1) node[below,black]  {$3$}; \draw[thick,black] (3,-0.5) -- (3,-1) node[below,black]  {$4$}; \draw[thick,black] (4,-0.5) -- (4,-1) node[below,black]  {$5$}; 
\draw[thick,blue] (5,0) -- (5,-1) node[below,black]  {$6$}; 
\draw[thick,black] (6,0) -- (6,-1) node[below,black]  {$7$};  
\end{tikzpicture}
\caption{Adapted noncrossing partition $\sigma_1 \cup \sigma_2$.}\label{aNC2}
\end{center}
\end{figure}
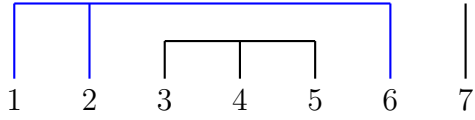

\begin{figure}
\begin{center}
\begin{tikzpicture}
\draw[thick,blue] (0,0) -- (1,0); 
\draw[thick,blue] (0,0) -- (0,-1) node[below,black]  {$1$}; \draw[thick,blue] (1,0) -- (1,-1) node[below,black]  {$2$}; 
\draw[thick,black] (2,0) -- (6,0); 
\draw[thick,black] (2,0) -- (2,-1) node[below,black]  {$3$}; \draw[thick,black] (3,0) -- (3,-1) node[below,black]  {$4$}; \draw[thick,black] (4,0) -- (4,-1) node[below,black]  {$5$}; 
\draw[thick,blue] (5,-0.5) -- (5,-1) node[below,black]  {$6$}; 
\draw[thick,black] (6,0) -- (6,-1) node[below,black]  {$7$};  
\end{tikzpicture}
\caption{Adapted noncrossing partition $\tau_1 \cup \tau_2$.}\label{aNC3}
\end{center}
\end{figure}
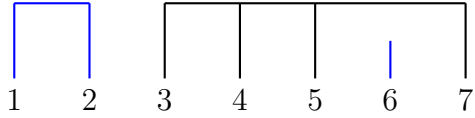

We get:

\begin{lem}
There exist (universal) coefficients $\alpha_{\pi^1,\pi^2}$ such that for any sequence $b_1\cdots b_n$ of random variables in $\mathcal A_1$ and $\mathcal A_2$
\begin{equation}
\label{adaptedfree}
	\phi_1\ast_f\phi_2(b_1\cdots b_n)
	=\sum\limits_{(\pi^1,\pi^2)\in ANC_n^S}	
	\alpha_{\pi^1,\pi^2}\Big(\prod_{\pi_j\in\pi^1}\phi_1(b_{\pi_j})\Big)
	\Big(\prod_{\pi_k\in\pi^2}\phi_2(b_{\pi_k})\Big).
\end{equation}
\end{lem}
The lemma follows by grouping in the word $b_1\cdots b_n$ the consecutive letters that belong to the same algebra to form an alternating tensor to which Remark \ref{closefo} can be applied.
Notice that, given any $(\pi^1,\pi^2)\in ANC^S_n$, one can recover the index set as $S=\bigcup_{\pi_j \in \pi^1} \pi_j$.
Such properties of universal free products and reasoning on them can be found already in Speicher \cite{speicher_90}. Details of the proof are left to the reader.

Let us point out for later use that when $\pi^1$ and $\pi^2$ contain each a single block, the coefficient $\alpha_{\pi^1,\pi^2}$ is obtained directly from the universal product formula in Definition \ref{defff} and is equal to one.
More precisely, $\pi^1,\pi^2$ can only both contain a single block in the cases that $b_1\cdots b_n = c_1\cdots c_id_{i+1}\cdots d_n$ with $c_k \in \mathcal A_p$ and $d_k\in \mathcal A_q$ (which corresponds to $n=2$, $a_1=c_1\cdots c_i,\ a_2=d_{i+1}\cdots d_n$ in Definition~\ref{defff})
or $b_1\cdots b_n = c_1\cdots c_id_{i+1}\cdots d_j e_{j+1}\cdots e_{n}$ with $c_k,e_k \in \mathcal A_p$ and $d_k\in \mathcal A_q$ (which corresponds to $n=3$, $a_1=c_1\cdots c_i,\ a_2=d_{i+1}\cdots d_j,\ a_3= e_{j+1}\cdots e_{n}$ in Definition~\ref{defff}),
where $p\not= q$.
We apply Definition~\ref{defff} in the first case with $I=\emptyset,\{1\},\{2\}$ and in the second case with $I=\{1,3\}$.

\begin{defn}[Universal monotone product]
\label{def:monounivprod}
For an alternating tensor $a_1\cdots a_n$,
$$
	\phi_1\blacktriangleright\phi_2(a_1\cdots a_n)
	:=\phi_1\big(\prod_{a_i\in {\mathcal A}_1}a_i\big) \prod_{a_j\in {\mathcal A}_2}\phi_2(a_j).
$$
\end{defn}

\begin{defn}[Universal anti-monotone product] For an alternating tensor $a_1\cdots a_n$,
$$
	\phi_1\blacktriangleleft\phi_2(a_1\cdots a_n)
	:=\Big(\prod_{a_i\in {\mathcal A}_1}\phi_1(a_i)\Big)\phi_2\big(\prod_{a_j\in {\mathcal A}_2}a_j\big).
$$
\end{defn}

As already mentioned, we omit the tensor case, which would be given by the product
$$
	\phi_1(\prod_{a_i\in {\mathcal A}_1}a_i)\phi_2(\prod_{a_j\in {\mathcal A}_2}a_j).
$$

In each of the above cases, if $X,Y$ are two random variables in $(\mathcal A,\varphi)$, we let $(\mathcal A_1,\varphi_1)$ and $(\mathcal A_2,\varphi_2)$ be two copies of $(\mathcal A,\varphi)$ embedded as usual in $\mathcal A_1\oast\mathcal A_2$ and $X_1,Y_2$ copies of $X$ and $Y$ respectively in $\mathcal A_1$ and $\mathcal A_2$. We set $X':=\iota_1(X_1)\in\mathcal A_1\oast\mathcal A_2$ and $Y'':=\iota_2(Y_2)\in \mathcal A_1\oast\mathcal A_2$. The two variables $X'$ and $Y''$ are independent for the five notions of independence and the respective products of states. The key difference between the five cases is the way the state map $\varphi$ is extended to $\mathcal A_1\oast\mathcal A_2$. For example, in the monotone case, $X'$ and $Y''$ are monotone independent for the state map $\phi_1 \blacktriangleright \phi_2$. 

For each of the mentioned probability theories, the (additive) convolution of two distributions of variables $X$, $Y$ belonging to $(\mathcal A,\varphi)$  is then defined as the distribution of the sum $X'+Y''$. For example, the moment of order $n$ associated to the Boolean, respectively monotone convolution of the distributions of $X$ and $Y$ is obtained by computing $\phi_1\bullet\phi_2((X'+Y'')^n)$, respectively $\phi_1\blacktriangleright\phi_2((X'+Y'')^n)$. In the first case (using notation from Section \ref{sec:notation}) we get:
$$
	\phi_1\bullet\phi_2((X'+Y'')^n)
	=\sum\limits_{S\subset [n]}\phi(X^{|S_1|})\cdots 
	\phi(X^{|S_l|})\phi(Y^{|J_1^S|})\cdots \phi(Y^{|J_k^S|}).
$$
This expression is symmetric (and equal to $\sum_{S\subset [n]}\phi(Y^{|S_1|})\cdots \phi(Y^{|S_l|})\phi(X^{|J_1^S|})\cdots \phi(X^{|J_k^S|})$): Boolean convolution is commutative.

In the second case:
$$
	\phi_1\blacktriangleright\phi_2((X'+Y'')^n)
	=\sum\limits_{S\subset [n]}\phi(X^{|S|})\phi(Y^{|J_1^S|})\cdots \phi(Y^{|J_k^S|}).
$$
Notice that, since the monotone and anti-monotone products are non-symmetric, the order matters and the convolution of the distributions of $X$ and $Y$ is different in general from the convolution of the distributions of $Y$ and $X$. 

We are now in the position to develop the three group structures corresponding to free, Boolean and monotone probabilities, by analogy with the commutative case. As the anti-monotone product is simply obtained by exchanging the roles of ${\mathcal A}_1$ and ${\mathcal A}_2$ in the free product of algebras, its treatment is omitted here. We recall first the framework that permits to obtain these products as the products arising from a group of characters on a (not necessarily commutative or cocommutative) Hopf algebra. In the following sections, this approach is linked to the one derived from universal products.


\section{Products of states}
\label{sect:GS}

Hereafter, as was done in the classical case, we extend the notion of \textit{state} (which we define here as a linear unital map on $\mathcal A$), to the notion of \textit{generalized measure}, i.e., an arbitrary unital linear form on the tensor algebra $T(\mathcal A)$.
Remark that our notion of state is sometimes called a \textit{formal state} (with the term state reserved for positive linear unital maps, which requires the assumption that $\mathcal{A}$ is a $*$-algebra that we do not assume here).

The generalized measure associated to a state $\varphi$ is the unital linear form, denoted $\phi: T(\mathcal A) \to \mathbb K$:
$$
	a_1\cdots a_n\longmapsto \phi(a_1\cdots a_n)
	=\varphi(a_1\cdot_{\!\!\scriptscriptstyle{\mathcal{A}}} \cdots \cdot_{\!\!\scriptscriptstyle{\mathcal{A}}}a_n).
$$
A generalized measure $\psi$ extends uniquely to a character $\Psi$ on the double tensor algebra $T(T^+(\mathcal A))$ by requiring that 
$$
	\Psi([w_1|\cdots |w_n]):=\psi(w_1)\cdots \psi(w_n).
$$
As $\psi$ is the restriction of $\Psi$ to $T^+(\mathcal A)$, this correspondence defines a bijection between generalized measures and characters on $T(T^+(\mathcal A))$. We denote $\mathcal G(\mathcal A)$ and $\mathcal G_T(\mathcal A)$ the set of generalized measures on $\mathcal A$ respectively the set of characters on $T(T^+(\mathcal A))$ and write
$$
	\Psi =: char(\psi) \quad \text{and} \quad \psi =:gm(\Psi)
$$
for the bijection from $\mathcal G(\mathcal A)$ to $\mathcal G_T(\mathcal A)$ and its inverse. 

Recall now the definition of a particular Hopf algebra structure on the double tensor algebra $T(T^+(\mathcal A))$.

\begin{defn}
The map $\Delta : T^+(\mathcal A) \to (\mathbb{K}\oplus T^+(\mathcal A)) \otimes   T(T^+(\mathcal A))$ is defined by
\begin{equation}
\label{HopfAlg}
	\Delta(a_1\cdots a_n) 
	:= \sum_{S \subseteq [n]} a_S \otimes  [a_{J_1^S} | \cdots | a_{J_k^S}]
	=\sum_{S \subseteq [n]} a_S \otimes [a_{J^S_{[n]}}],
\end{equation} 
with $\Delta(\un):= \un \otimes \un$ and $[a_{J^S_{[n]}}]:=[a_{J_1^S} | \cdots | a_{J_k^S}]$. It is then extended multiplicatively to a coproduct on $T(T^+(\mathcal A))$
$$
	\Delta([w_1 | \cdots | w_m]) := \Delta(w_1) \cdots \Delta(w_m).
$$
\end{defn}
 
\begin{thm}[\cite{ebrahimipatras_15}]
The graded algebra $T(T^+(\mathcal A))$ equipped with the coproduct \eqref{HopfAlg} is a graded connected non-commutative and non-cocommutative Hopf algebra. 
\end{thm}

A crucial observation is that the coproduct \eqref{HopfAlg} can be split into two parts as follows. On $T^+(\mathcal A)$ define the {\it{left half-coproduct}} by
\begin{equation*}
	\Delta^+_{\prec}(a_1 \cdots a_n) 
	:= \sum_{1 \in S \subseteq [n]} a_S \otimes [a_{J^S_{[n]}}],
\end{equation*}
and
\begin{equation*}
	\Delta_{\prec}(a_1 \cdots a_n) 
	:= \Delta^+_{\prec}(a_1 \cdots a_n) - a_1 \cdots a_n \otimes \un. 
\end{equation*}
The {\it{right half-coproduct}} is defined by
\begin{equation*}
	\Delta^+_{\succ}(a_1 \cdots a_n) 
	:= \sum_{1 \notin S \subset [n]} a_S \otimes [a_{J^S_{[n]}}],
\end{equation*}
and
\begin{equation*}
	\Delta_{\succ}(a_1 \cdots a_n) 
	:= \Delta^+_{\succ}(a_1 \cdots a_n) -  \un \otimes a_1 \cdots a_n.
\end{equation*}
Which yields $\Delta = \Delta^+_{\prec} + \Delta^+_{\succ}$, and $\Delta(w) = \Delta_{\prec}(w) + \Delta_{\succ}(w) + w \otimes \un + \un \otimes w.$ This is extended to $T^+(T^+(\mathcal A))$ by defining
\begin{align*}
	\Delta^+_{\prec}([w_1 | \cdots | w_m]) 
	&:= \Delta^+_{\prec}(w_1)\Delta(w_2) \cdots \Delta(w_m) \\
	\Delta^+_{\succ}([w_1 | \cdots | w_m]) 
	&:= \Delta^+_{\succ}(w_1)\Delta(w_2) \cdots \Delta(w_m). 
\end{align*}

By duality, the dual vector space $T(T^+(\mathcal A))^\ast$ of linear forms on $T(T^+(\mathcal A))$ is an algebra with respect to the convolution product defined for $f,g \in T(T^+(\mathcal A))^\ast$ by
\begin{equation}
\label{convProd}
	f * g := (f\otimes g)\circ \Delta,
\end{equation} 
where we implicitly identify $\mathbb K\otimes\mathbb K$ with $\mathbb K$. The unit for the product \eqref{convProd} is the canonical projection, denoted $\nu$ hereafter, from $T(T^+(\mathcal A))$ to $\mathbb K$ (sending both non-empty words $a_1\cdots a_n$ and elements $[w_1|\cdots|w_k]$, $k>1$, in $T(T^+(\mathcal A))$ to zero). We summarize the group-theoretical implications of these ideas in the

\begin{prop}
Characters in $\mathcal G_T(\mathcal A)$ form a group under convolution \eqref{convProd}. The corresponding Lie algebra $g_T({\mathcal A})$ of infinitesimal characters is characterised by $\alpha(\mathbf 1) =0$ and $\alpha([w_1|\cdots |w_n])=0$, for $\alpha \in g_T({\mathcal A})$ and $w_1,\ldots,w_n \in T^+(A)$, $n \geq 2$.  $\mathcal G_T(\mathcal A)$ and $g_T({\mathcal A})$ are related bijectively by the convolution exponential and logarithm, $\exp^*$ respectively $\log^*$.  \end{prop}

The {\it{left}} and {\it{right half-convolution}} or {\it{half-shuffle}} products on $T^+(T^+(\mathcal A))^\ast$ are obtained by dualizing the half-coproducts:
\begin{equation}
\label{halfshuffle}
	f\prec g:=(f\otimes g)\circ \Delta_\prec 
	\quad\
	{\rm{and}}
	\quad\
	f\succ g:=(f\otimes g)\circ \Delta_\succ,
\end{equation} 
which split the associative convolution product  \eqref{convProd}
\begin{equation*}
	f * g = f \succ g + f \prec g. 
\end{equation*} 	
The two operations in \eqref{halfshuffle} are not associative. Indeed, they satisfy the so-called {\it{half-shuffle identities}} defining a shuffle algebra:
\begin{eqnarray}
	(f \prec g)\prec h &=& f \prec(g * h)        		\label{A1}\\
  	(f \succ g)\prec h &=& f \succ(g\prec h)   		\label{A2}\\
   	f \succ(g\succ h)   &=& (f  * g)\succ h.	      	\label{A3}
\end{eqnarray}
Relations \eqref{A1}-\eqref{A3} are extended using Sch\"utzenberger's trick, that is, by setting $\nu\succ f:=f$, $f \prec \nu:=f$, $\nu \prec f:=0$ and $f\succ \nu:=0$ for $f \in T^+(T^+(\mathcal A))^\ast$ ($\nu\prec \nu$ and $\nu\succ \nu$ are left undefined, whereas $\nu\ast \nu=\nu$). 

Recall finally (e.g.~from \cite{ebrahimipatras_17}) the following proposition. A  (left) pre-Lie algebra \cite[Chap.~6]{cp2021} is a vector space $V$ equipped with a bilinear product $\shortmid\!\!\rhd$ such that the (left) pre-Lie identity 
$$
	(x \!\shortmid\!\!\rhd y)\!\shortmid\!\!\rhd z
	-x \!\shortmid\!\!\rhd (y \!\shortmid\!\!\rhd z)
	=(y \!\shortmid\!\!\rhd x)\!\shortmid\!\!\rhd z
	-y \!\shortmid\!\!\rhd (x\!\shortmid\!\!\rhd z)
$$
is satisfied. The notion of right pre-Lie algebra is defined analogously. The bracket $[x,y]:=x\!\shortmid\!\!\rhd y - y\!\shortmid\!\!\rhd x$ satisfies the Jacobi identity and defines a Lie algebra structure on $V$. The next statement is a consequence of the fact that any shuffle algebra carries a (left) pre-Lie product defined by antisymmetrization of the half-shuffle products.

\begin{prop} \label{prelieprop}
The space $T^+(T^+(\mathcal A))^\ast$ is equipped with a (left) pre-Lie algebra structure defined by the product
\begin{equation*}
	f \!\shortmid\!\!\rhd g := f \succ g  - g \prec f
\end{equation*}
and the associated Lie bracket $[f,g] = 	f \!\shortmid\!\!\rhd g  - 	g \!\shortmid\!\!\rhd f$ is equal to $f*g-g*f$.
\end{prop}

For $\alpha \in g_T({\mathcal A})$ the {\it{left}} and {\it{right half-shuffle}}, or ``time-ordered'', exponentials are defined by (see \cite{ebrahimipatras_15} for details)
\begin{equation*}
	\mathcal{E}_\prec(\alpha) := \exp^{\prec}(\alpha) :=\nu + \sum_{n > 0} \alpha^{\prec{n}} 
\end{equation*}
\begin{equation*}
	\mathcal{E}_\succ(\alpha) := \exp^{\succ}(\alpha):=\nu + \sum_{n > 0}  \alpha^{\succ n}, 
\end{equation*}
where $\alpha^{\prec{n}} := \alpha \prec(\alpha^{\prec{n-1}})$, $\alpha^{\prec{0}}:=\nu$ and $\alpha^{\succ{n}} := (\alpha^{\succ{n-1}})\succ \alpha$, $\alpha^{\succ{0}}:=\nu$. They satisfy by definition the fixed point equations
\begin{equation}
\label{recursion}
	\mathcal{E}_\prec(\alpha)=\nu + \alpha \prec \mathcal{E}_\prec(\alpha) 
	\qquad 
	\mathcal{E}_\succ(\alpha)= \nu + \mathcal{E}_\succ(\alpha) \succ \alpha,
\end{equation}
and provide bijections between the Lie algebra $g_T({\mathcal A})$ and the group of characters $\mathcal G_T(\mathcal A)$. Its inverses, known as {\it{left}} and {\it{right half-shuffle}} logarithms, $\mathcal{L}_\prec$ respectively $\mathcal{L}_\succ$, can be deduced from the fixed point equations \eqref{recursion} using the half-shuffle identities \eqref{A1}-\eqref{A3}:
\begin{equation}
\label{halfshufflelog}
	\mathcal{L}_\prec(\Phi)=(\Phi - \nu) \prec \Phi^{-1} 
	\quad \text{and} \quad
	\mathcal{L}_\succ(\Phi)=\Phi^{-1} \succ (\Phi - \nu).
\end{equation}

With this structure in place, the shuffle algebra approach \cite{ebrahimipatras_19} permits to express moment-cumulant relations in non-commutative probability in terms of the convolution or shuffle exponential and logarithm, $\exp^*$ respectively $\log^*$, and the half-shuffle exponentials and logarithms defined in equations \eqref{recursion} and \eqref{halfshufflelog}. In addition, the unified setting covers relations between the different families of cumulants \cite{Arizmendi_15} in terms of a single fundamental identity in the group $\mathcal G_T(\mathcal A)$
\begin{equation*}
	\Phi=\exp^*(\rho)=\mathcal{E}_\prec(\kappa) = \mathcal{E}_\succ(\beta),
\end{equation*}
and the corresponding infinitesimal relation in $g_T({\mathcal A})$ \cite{CelEbraPatPer2022,ebrahimipatras_17}. Here $\Phi \in \mathcal G_T(\mathcal A)$ and $\rho$, $\kappa$, and $\beta$ are infinitesimal characters in $g_T({\mathcal A})$, which correspond respectively to monotone, free and boolean cumulants.


\section{The monotone group of states}
\label{sec:monotonestate}

The previous two sections lead to the following definitions. We refer the reader to reference \cite{ebrahimipatras_15} for details (in particular the proof that the following constructions define groups -- the next one is actually a rephrasing in terms of generalized measures of definitions given in \cite{ebrahimipatras_15} and other articles following the same approach).

Recall from Section~\ref{sect:GS} that $\mathcal G(\mathcal A)$ denotes the set of generalized measures on $\mathcal{A}$ and that there is a bijection $gm\colon \mathcal{G}_T(\mathcal{A})\to \mathcal{G}(\mathcal{A})$, where $ \mathcal{G}_T(\mathcal{A})$
is the set of characters on $T(T^+(\mathcal A))$.

\begin{defn}[Monotone group of states, algebraic definition]\label{demonogr}
The set $\mathcal G(\mathcal A)$ equipped with the product
$$
	\phi\ast \psi:=gm(char(\phi)\ast char(\psi))
$$
is a group called the {\it monotone group of states}.
\end{defn}

In the next proposition, the notation of Section \ref{sect:up} is used. We write $\mathcal B_1$ and $\mathcal B_2$ for two copies of the tensor algebra $T(\mathcal A)$ and $\phi_1$, $\psi_2$ for two copies of elements of $\phi,\psi \in \mathcal G(\mathcal A)$ acting respectively on $\mathcal B_1$ and $\mathcal B_2$. Given an element $a$ in $\mathcal A$, we write $a'$ and $a^{\prime\prime}$ for copies of $a$ in $\mathcal B_1$ respectively $\mathcal B_2$, and use the canonical embeddings $\iota_1$ and $\iota_2$ to view them as elements of the product $\mathcal B_1\oast\mathcal B_2$.

\begin{prop}[Monotone group of states, probabilistic definition]
The set $\mathcal G(\mathcal A)$ of generalized measures equipped with the product
\begin{equation}\label{mogrsta}
	\phi\ast \psi(a_1\cdots a_n)
	:=\phi_1\blacktriangleright \psi_2((a_1'+a_1^{\prime\prime})\cdots (a_n'+a_n^{\prime\prime})), 
\end{equation}
is a group that identifies with the monotone group of states of Definition \ref{demonogr}.
\end{prop}

\begin{proof}
The identity between the two definitions follows from the definition of the coproduct $\Delta$ on $T(T^+(\mathcal A))$ and the universal monotone product $\blacktriangleright$ given in Definition \ref{def:monounivprod}. The definition of $\Delta$ indeed immediately leads, for the convolution product defined by Equation \eqref{convProd} to the identity
$$
	\phi\ast \psi(a_1\cdots a_n)
	=\sum\limits_{S\subset [n]}\phi(a_S)\psi(a_{J_1^S}) \cdots \psi(a_{J_k^S}).
$$
In the second case (identity \eqref{mogrsta}), the same  identity follows by an elementary enumerative combinatorics argument  that we omit -- it amounts to parametrizing by subsets of $[n]$ the products of $a_i'$s and $a_j^{\prime\prime}$s appearing in the expansion of $(a_1'+a_1^{\prime\prime})\cdots (a_n'+a_n^{\prime\prime})$.
\end{proof}

Let us redo now the analysis of the classical case in a non-commutative setting. Assume that $\mathcal A$ has a basis $\mathcal B=(b_i)_{i\in\NN}$. Regarding this assumption, recall that a wide class of non-commutative probability spaces are subalgebras of algebras of operators on a Hilbert space and have a dense subspace spanned by a Hilbert space basis.

Given a sequence $I=(i_1,\ldots,i_k)\in{\mathbb N}^k$, we set $b_I:=b_{i_1}\cdots b_{i_k}$.
Notice that $b_I$ is now a word in the $b_i$ (and not anymore a commutative monomial as in the classical case). The set $\mathcal Y$ of words $b_I$ defines then a (complete) system of coordinates on $\mathcal G(\mathcal A)$, that is, for $\phi\in\mathcal G(\mathcal A)$,
$$
	\langle\phi |b_I\rangle:= \phi(b_I).
$$
The definition of the group law on $\mathcal G(\mathcal A)$ implies that the $b_I$ are representative functions on the group \cite[2.7 Representative functions]{cp2021}. We get the alternative presentation of $\mathcal G(\mathcal A)$ as the group of characters of its Hopf algebra of representative functions generated by the elements of $\mathcal Y$:

\begin{cor}
The monotone group of states is canonically isomorphic to the group, denoted $\mathcal G_S(\mathcal A)$ of characters of the Hopf algebra $S(T^+(\mathcal A))$ equipped with the product of polynomials over $T^+(\mathcal A)$ and the coproduct whose restriction to the generators of the polynomial algebra $S(T^+(\mathcal A))$ is given by
$$
	\Delta_m \colon T^+(\mathcal{A}) \to (\mathbb{K} \oplus T^+(\mathcal{A})) \otimes  S(T^+(\mathcal{A}))
$$
\begin{equation}
\label{def:coproducts1}
	\Delta_m(a_1 \cdots a_n) 
	:= \sum_{S \subseteq [n]} a_S \otimes  (a_{J^S_1} | \cdots | a_{J^S_k}).
\end{equation} 
We call $\Delta_m$ the monotone coproduct. 
\end{cor}

For example, let $a_1a_2a_3 \in T^+(\mathcal{A})$ and let us calculate the reduced monotone coproduct (that is, $\overline\Delta_m(x):=\Delta_m(x)-x\otimes \un-\un\otimes x$):
\begin{align*}
	\overline\Delta_m(a_1 a_2 a_3) 
	&=  a_2 a_3 \otimes a_1 + a_1 a_2 \otimes a_3 + a_1 a_3 \otimes a_2\\
	&\quad + a_1  \otimes a_2a_3 + a_2  \otimes (a_1 | a_3) +  a_3  \otimes a_1 a_2.  
\end{align*}

Notice the fact that $S(T^+(\mathcal A))$ is indeed a Hopf algebra (that is, the coproduct is coassociative as well as a map of associative unital algebras and $S(T^+(\mathcal A))$ has an antipode) can be checked easily by hand -- the proof is similar to the one for $T(T^+(\mathcal A))$, as defined above. However, it also follows automatically from the group structure of $\mathcal G(\mathcal A)$; for general properties of (pro-affine) algebraic groups, and the bijection between the group and linear forms on $T^+(\mathcal A)$ (or, equivalently, functions on the chosen set of coordinates $\mathcal Y$), see e.g.~\cite[Sect.~3.5 and 3.6]{cp2021}.

The coproduct map \eqref{def:coproducts1} has a particular property: it is left-linear, that is, it maps the vector space $T^+(\mathcal{A}) $ of the Hopf algebra generators of $S(T^+(\mathcal{A}))$ as a polynomial algebra to $(\un \oplus T^+(\mathcal{A})) \otimes  S(T^+(\mathcal{A}))$. This property ensures that the following proposition holds (see e.g.~\cite[6.4 Left-Linear Groups and Fa\`a di Bruno]{cp2021}).

Let the map, called later on the linearised coproduct map, 
\begin{equation}
\label{lincoprod1}
	\delta_m \colon T^+(\mathcal{A}) \to T^+(\mathcal{A}) \otimes T^+(\mathcal{A})
\end{equation} 
be defined in terms of the monotone coproduct \eqref{def:coproducts1}, 
$$
	\delta_m := (p \otimes p) \Delta_m,
$$ 
where the map $p \colon S(T^+(\mathcal{A})) \to T^+(\mathcal{A})$ is the projector that is zero on $\un $ as well as on products in $S^+(T^+(\mathcal{A}))$, i.e., elements $(w|v)$ with $w \in T^+(\mathcal{A})$ and $v \in S^+(T^+(\mathcal{A}))$ are mapped to zero, and it is the identity on $T^+(\mathcal{A})$. The  linearised coproduct can be described explicitly. Indeed, for any word $w=a_1 \cdots a_n \in T^+(\mathcal{A})$ we have
\begin{equation*}
	\delta_m(w)=\sum_{\emptyset \neq S_c\subsetneq [n]} a_{[n]-S_c} \otimes a_{S_c},
\end{equation*} 
where $\emptyset \neq S_c\subsetneq [n]$ denotes a non-empty subset which has a single connected component, that is, a non-empty interval in $[n]$. For instance 
\begin{align*}
	\delta_m(a_1 a_2 a_3) 
	&=  a_2 a_3 \otimes a_1 + a_1 a_2 \otimes a_3 + a_1 a_3 \otimes a_2
	 	+ a_1  \otimes a_2a_3 +  a_3  \otimes a_1 a_2\\
	\delta_m(a_1 a_2 a_3a_4) 
	&=  a_2 a_3a_4 \otimes a_1 + a_1a_3a_4 \otimes a_2 + a_1a_2a_4 \otimes a_3 + a_1 a_2a_3 \otimes a_4 
		+ a_1  \otimes a_2a_3a_4 \\
	&\quad +  a_4  \otimes a_1 a_2a_3 + a_1a_2  \otimes a_3a_4 
		+ a_3a_4  \otimes a_1a_2 +a_1a_4  \otimes a_2a_3.  
\end{align*}
Observe that $\delta_m$ is neither coassociative nor cocommutative. Instead, we find the following.
 
\begin{prop}\label{prop:coPreLie}
The linearised coproduct \eqref{lincoprod1} satisfies the left pre-Lie coalgebra identity
\begin{equation}
\label{coprelie}
	\mathrm{a}_m = (\tau \otimes \id)\mathrm{a}_m,    
\end{equation} 
where $\mathrm{a}_m:=(\delta_m \otimes \id)\delta_m - (\id \otimes \delta_m)\delta_m$.
The map $\tau$ is defined by switching tensor products, $\tau(x\otimes y):=y\otimes x$.
\end{prop}

For completeness, let us conclude the analysis of the monotone group of states with its Lie theoretic part. Once again, the following results are a consequence of the fact that $S(T^+(\mathcal A))$, as a polynomial Hopf algebra, is graded and connected (its degree zero component is the ground field) and left linear (this last property being relevant only for the assertions on the connections between the Lie and pre-Lie structures). These results are strictly similar to those holding for the group of characters on $T(T^+(\mathcal A))$ and its pre-Lie and Lie algebra of infinitesimal characters, and follow for the same reasons (see e.g.~\cite{ebrahimipatras_17}). They are worth being stated, but we omit proofs. 

A linear form $\kappa \in S(T^+(\mathcal A))^\ast$ is called an infinitesimal character, if $\kappa(\mathbf{1})=0$ (recall that $\mathbf 1$ is the unit of the product of $S(T^+(\mathcal A))$, and if $\kappa(w_1|\cdots|w_n)=0$ for all $w_1,\ldots,w_n \in T^+(\mathcal A)$, $n\geq 2$. Let us write $g(\mathcal A)$ for the set of linear forms $\mu$ on $T^+(\mathcal A)$ such that $\mu(\mathbf{1})=0$ and $g_S(\mathcal A)$ for the set of infinitesimal characters on $S(T^+(\mathcal A))$. There is an obvious bijection between the two sets: the restriction map from $S(T^+(\mathcal A))^\ast$ to $T^+(\mathcal A)^\ast$ defines for example the bijection from $g_S(\mathcal A)$ to $g(\mathcal A)$. 

Infinitesimal characters are closed under commutator brackets defined on $S(T^+(\mathcal A))^\ast$ in terms of the convolution product $\star_m$ on $S(T^+(\mathcal A))^\ast$ dual to the coproduct $\Delta_m$ on $S(T^+(\mathcal A))$
\begin{equation}
\label{Lie1}
	[\alpha , \beta] := \alpha \star_m \beta - \beta \star_m \alpha,
\end{equation} 
for $\alpha,\beta \in g_S(\mathcal A)$. Hence, $g_S(\mathcal A)$ is a Lie algebra. We call $\star_m$ the monotone convolution product (on $S(T^+(\mathcal A))^\ast$). Since $\alpha(\mathbf{1}) =0$ for $\alpha \in g_S(\mathcal A)$, the exponential defined by its power series with respect to monotone convolution, $\exp^{\star_m}(\alpha)(w) := \nu + \sum_{j > 0} \frac{1}{j!}\alpha^{{\star_m} j}(w)$, is a finite sum terminating at $j=\deg(w)$, and defines a bijection from $g_S(\mathcal A)$ onto $\mathcal G_S(\mathcal A)$. The compositional inverse of $\exp^{\star_m}$ is given by the logarithm, $\log^{\star_m}( \nu + \gamma)(w)=\sum_{k \ge 1}\frac{(-1)^{k-1}}{k}\gamma^{{\star_m}k}(w)$, where $\gamma \in g_S(\mathcal A)$ and $\nu$ is the canonical augmentation map from $S(T^+(\mathcal A))$ to the ground field.
Again the sum terminates at $k=\deg(w)$ for any $w \in H$ as $\gamma(\mathbf{1}) =0$. This is a general phenomenon \cite[Sect.~3.5]{cp2021}.

The exponential $\exp^{\star_m}$ can be computed explicitly:

\begin{thm}
Let $(\mathcal A,\varphi)$ be a non-commutative probability space with unital linear map $\varphi \colon \mathcal A \to \mathbb{K}$. Let $\Phi$ denote its extension to $S(T^+(\mathcal A))$ as a character. Let $\rho$ be the infinitesimal character defined as $\rho:=\log^{\star_m}(\Phi)$. For a word $w=a_1 \cdots a_n \in T^+(\mathcal A)$, we have the relation
\begin{equation}
\label{monotone2}
	\Phi(w)=\exp^{\star_m}(\rho)(w) 
		= \sum_{j=1}^n \frac{\rho^{\star_m j}(w)}{j!} 
		= \sum\limits_{\pi \in NC_n} \frac{1}{\mathrm{t}(\pi)!} \text{h}_{\pi}(a_1, \ldots, a_n).
\end{equation}
\end{thm}

The tree factorial $\mathrm{t}(\pi)!$ corresponds to the forest $\mathrm{t}(\pi)$ of rooted trees encoding the nesting structure of the non-crossing partition $\pi \in NC_n$, see \cite{Arizmendi_15} for definitions, and  $h_\pi(a_1, \ldots, a_n):=\prod_{\pi_i \in \pi} \rho(a_{\pi_i})$. 

\begin{proof}
We omit the proof of the theorem. Indeed, the calculation is the same as the one proving the identity when $\star_m$ is replaced by the convolution product $\ast$ of linear forms on $T(T^+(\mathcal A))$ (see \cite{ebrahimipatras_17}). This follows from the fact that the coproducts \eqref{HopfAlg} and \eqref{def:coproducts1} differ only by the fact that the terms on the right hand side belong respectively to $T(T^+(\mathcal A))$ and $S(T^+(\mathcal A))$. 
\end{proof}

\begin{cor}
Equation \eqref{monotone2} reproduces the monotone moment-cumulant relation of \cite{hasebesaigo_11}
\begin{equation*}
	\Phi(w)= \sum\limits_{\pi \in NC_n} \frac{1}{\mathrm{t}(\pi)!} \text{h}_{\pi}(a_1, \ldots, a_n).
\end{equation*}
This implies that $\text{h}_n(a_1,\ldots,a_n):=\rho(a_1 \cdots a_n)$ identifies with the $n$th multivariate monotone cumulant maps from~\eqref{eq:monotoneMCrel}.
\end{cor}

\begin{cor}
The equation also implies that the bijection $\exp^{\star_m}$  between  $g_S(\mathcal A)$ and $\mathcal G_S(\mathcal A)$  and its inverse $\log^{\star_m}$ identify (up to canonical isomorphisms) with the bijection $\exp^\ast$, respectively its inverse $\log^\ast$ between $g_T(\mathcal A)$ and $\mathcal G_T(\mathcal A)$, the Lie algebra of infinitesimal characters, respectively group of characters in $T(T^+(\mathcal A))^\ast$ (see \cite{ebrahimipatras_17}).
\end{cor}

Going back to Proposition \ref{prop:coPreLie}, we can dualise the pre-Lie coproduct \eqref{lincoprod1} and obtain a left pre-Lie product on $g(\mathcal A)$
\begin{equation}
\label{preLie1}
	\alpha \rhd \beta := (\alpha \otimes \beta) \circ \delta_m,
\end{equation} 
which satisfies the dual of \eqref{coprelie}, i.e., the left pre-Lie relation
\begin{equation*}
	(\alpha \rhd \beta) \rhd \gamma - \alpha \rhd (\beta \rhd \gamma )
	= ( \beta \rhd \alpha) \rhd \gamma - \beta \rhd (\alpha \rhd \gamma).
\end{equation*} 
By antisymmetrization, we have another Lie bracket on $g(\mathcal A)$
\begin{equation}
\label{Lie2}
	\llbracket \alpha , \beta \rrbracket := \alpha \rhd \beta - \beta \rhd \alpha.
\end{equation} 

\begin{prop}
The space $g(\mathcal A)$ is closed under the pre-Lie product \eqref{preLie1}, i.e., $(g(\mathcal A),\rhd)$ is a pre-Lie algebra. 
Furthermore, on the space $g(\mathcal A)\cong g_S(\mathcal A)$ the two Lie brackets \eqref{Lie1} and \eqref{Lie2} coincide, that is, $\llbracket \alpha , \beta \rrbracket = [ \alpha, \beta ]$, for $ \alpha, \beta \in g(\mathcal A)$.   
\end{prop}

The proposition should be related to Proposition \ref{prelieprop}. Here, an important remark is in order: the pre-Lie products defined in these two propositions are related but different. However, they give rise to the same Lie algebra. This reflects the fact that some tools that are available when working with $T(T^+(\mathcal A))$, such as the half-shuffle products, are not available when working in what we call here the ``probabilistic approach'', i.e., in the context of $S(T^+(\mathcal A))$.


\section{The Boolean group of states}
\label{sect:bg}

Seen from the Hopf algebraic point of view, the construction of a group of states in the Boolean and free probability cases results from a different approach than the one used in the monotone case. 

Let us be more precise. We first introduce another notation. Given a linear map $\alpha\in T^+(\mathcal A)^\ast$, we extend it to an infinitesimal character denoted $ichar(\alpha)$ on $T(T^+(\mathcal A))$. For $k \not=1$, it maps by definition the components $T^+(\mathcal A)^{\otimes k}$ of $T(T^+(\mathcal A))$ to zero. As such the extension is well defined. Conversely, we write $igm$ for the inverse map from infinitesimal characters to elements of $T(\mathcal A)^\ast$.
The following definition and its motivation can be found in reference \cite{ebrahimipatras_19}. Recall the definition of the two half-shuffle exponentials and logarithms in Section \ref{sect:GS}.

\begin{defn}[Boolean group of states, algebraic definition]\label{defnBglawalg}
The set $\mathcal G(\mathcal A)$ equipped with the product
$$
	\phi \ogreaterthan \psi
	:=gm( \mathcal{E}_\succ(\mathcal{L}_\succ(char(\phi))+\mathcal{L}_\succ(char(\psi))) )
$$
is a group called the {\it Boolean group of states}.
\end{defn}

Recall also the Boolean moment-cumulant relation~\eqref{eq:booleanMCrel}, which we rewrite as
\begin{equation}
\label{eq:booleanc}
			\phi(a_1\cdots a_n)=\sum_{\pi \in \operatorname{Int}_n} \text{b}_\pi(a_1 \cdots a_n),
\end{equation}
where as usual $\text{b}_\pi(a_1 \cdots a_n):=\prod_{\pi_k \in \pi} \text{b}_{|\pi_k|}(a_{\pi_k})$.
Recall that $\operatorname{Int}_n$ stands for the set of Boolean (i.e.~interval) partitions of $[n]$. 
Given a non-commutative probability space $(\mathcal A,\varphi)$ with the usual extension of $\varphi$ to $\phi \in T^+(\mathcal A)^\ast$, relations~\eqref{eq:booleanc} define implicitly the Boolean cumulants $\{\text{b}_n\}_{n>0}$ in terms of $\phi$ respectively the state map $\varphi$.
By M\"obius inversion,
\begin{equation}
\label{boolean2}
			\text{b}_n(a_1\cdots a_n)=\sum_{\pi \in \operatorname{Int}_n}(-1)^{|\pi|-1} \phi_\pi(a_1 \cdots a_n).
\end{equation}
Again, $\phi_\pi(a_1 \cdots a_n):=\prod_{\pi_k \in \pi} \phi(a_{\pi_k})$. Notice for later use that Equation \eqref{boolean2} shows that $\text{b}_n(a_1\cdots a_n)=\phi(a_1 \cdots a_n)+\text{l.o.t.}$, where $\text{l.o.t.}$ refers to a sum of products of evaluations of $\phi$ on words of length strictly less than $n$. Recall also from \cite{ebrahimipatras_17} the following statement:

\begin{prop}
With the same notation as above, we have
\begin{equation*}
	\phi =gm( \mathcal{E}_{\succ} (ichar(\mathrm{b}))),
\end{equation*}
\begin{equation}
\label{def:Bcumul2}
	\mathrm{b} =igm( \mathcal{L}_{\succ} (char(\phi))).
\end{equation}
\end{prop}

Let us translate these results in the language of algebraic groups. We fix a basis  $\mathcal B=(b_i)_{i \in\NN}$ of $\mathcal A$. As previously, given a sequence $I=(i_1,\ldots,i_k)\in\mathbb N^k$, we set $b_I:=b_{i_1} \cdots b_{i_k}\in T^+(\mathcal A)$. The set $\mathcal Y$ of words $b_I$ defines then a (complete) system of coordinates on $\mathcal G(\mathcal A)$. 
The definition of the Boolean group law on $\mathcal G(\mathcal A)$ implies that $\phi \ogreaterthan \!\psi(b_I)$ can be expanded as a linear combination of products of terms $\phi(b_J)$ and $\psi(b_K)$, where $J$ and $K$ stand for subsets of $I$. In other words, the $b_I$ are representative functions on the group and $\mathcal G(\mathcal A)$ is the group of characters of its Hopf algebra of representative functions generated by the elements of $\mathcal Y$ (see again \cite{cp2021} for details). From \eqref{boolean2} and the remark that follows the equation, the monomials
$$
	K_{b,I}:=\sum_{\pi \in \operatorname{Int}_n}(-1)^{|\pi|-1} \prod\limits_{\pi_i\in\pi}b_{\pi_i}\in S(T^+(\mathcal A))
$$
form another system of coordinates of the group that we call the {\it Boolean coordinates} and denote by $\mathcal {B}ool$.\footnote{As pointed out to us by the referee, they can be understood as Boolean analog of (classical) generalized $k$-statistics, advocated for example by T.~P.~Speed \cite{speed}. The same remark applies in the free case presented in the next section.} From the group law $\ogreaterthan$ in Definition \ref{defnBglawalg} together with \eqref{def:Bcumul2}, we obtain
$$
	\mathcal{L}_{\succ}(char(\phi \ogreaterthan \!\psi))
	=\mathcal{L}_{\succ}(char(\phi))  +\mathcal{L}_{\succ}(char(\psi))
$$
and
$$
	\phi \ogreaterthan \!\psi (K_{b,I})=\phi(K_{b,I})+\psi(K_{b,I}).
$$
Hence, in the system of Boolean coordinates, the group law is additive and we get the alternative presentation of the group $(\mathcal G(\mathcal A),\ogreaterthan)$:

\begin{cor}\label{coprolinB}
The Boolean group of states is, up to a canonical isomorphism, the group of characters of the Hopf algebra $S({\mathcal B}ool)\cong S(T^+(\mathcal A))$ equipped with the product of polynomials over ${\mathcal B}ool$ and the coproduct 
$$
	\Delta_B(K_{b,I}):=K_{b,I} \otimes \mathbf 1 + \mathbf 1 \otimes K_{b,I}.
$$
\end{cor}

In the following proposition, we use again freely the notation of Section \ref{sect:up}. We write $\mathcal B_1$ and $\mathcal B_2$ for two copies of the tensor algebra $T(\mathcal A)$ and $\phi_1$, $\psi_2$ for two copies of elements $\phi, \psi \in \mathcal G(\mathcal A)$ acting respectively on $\mathcal B_1$ and $\mathcal B_2$. Given an element $a$ in $\mathcal A$, we write $a'$ and $a^{\prime\prime}$ for copies of $a$ in $\mathcal B_1$ respectively $\mathcal B_2$, both naturally embedded into the product $\mathcal B_1\oast\mathcal B_2$

\begin{prop}[Boolean group of states, probabilistic definition]
\label{idprodb}
The set $\mathcal G(\mathcal A)$ equipped with the product
$$
	\phi\ogreaterthan \!\psi(a_1\cdots a_n):=\phi_1 \bullet \psi_2((a_1'+a_1'')\cdots (a_n'+a_n'')), 
$$
is a group that identifies with the Boolean group of states defined previously.
\end{prop}

We will decompose the proof of the proposition into several steps.
In the following arguments, we will use only the fact that the probabilistic definition of the product implies its associativity (the existence of a group structure will follow later from the identification of the two definitions of the product).
Let us first compute explicitly the product law of $\mathcal G(\mathcal A)$ as defined in Proposition \ref{idprodb} in the coordinates $\mathcal Y$, that is, the corresponding coproduct that we denote temporarily $\Delta_b$. From the definition of the product $\bullet$ of states, we obtain the cocommutative coproduct map:
$$
	\Delta_b  \colon T^+(\mathcal{A}) \to S(T^+(\mathcal{A})) \otimes S(T^+(\mathcal{A})),
$$ 
\begin{equation}
\label{coprodbool}
	\Delta_b(a_1 \cdots a_n) 
	:= \sum_{S \subseteq [n]} (a_{S_1}  | \cdots | a_{S_l}) \otimes(a_{J^S_1} | \cdots | a_{J^S_k}).
\end{equation} 
It is extended multiplicatively to the polynomial algebra $S(T^+(\mathcal{A}))$ by $\Delta_b(\mathbf{1}):= \mathbf{1} \otimes \mathbf{1}$ and
$$
	\Delta_b(w_1 | \cdots | w_n) := \Delta_b(w_1) \cdots \Delta_b(w_n).
$$
For example, the reduced coproduct action on the word $a_1 a_2 a_3$ yields:
\begin{align*}
	\overline\Delta_b(a_1 a_2 a_3)
	&=  a_2 a_3 \otimes a_1 + a_1 a_2 \otimes a_3 + (a_1 | a_3) \otimes a_2\\
	&\quad + a_1  \otimes a_2a_3 + a_2  \otimes (a_1 | a_3) +  a_3  \otimes a_1 a_2.  
\end{align*}

We write $H_b$ for the Hopf algebra $S(T^+(\mathcal{A}))$ equipped with the product of polynomials and with the coproduct defined in \eqref{coprodbool}. The corresponding convolution product of linear forms on $H_b$ is denoted $\star_b$. 

\begin{lem}\label{lem:boolean_inf_char}
Let $\alpha_1,\ldots,\alpha_k$ be infinitesimal characters on $H_b$ and $w=a_1\cdots a_n \in T(\mathcal{A})$ a word. Then, we have:
\begin{equation*}
	\scal{\alpha_1\star_b\cdots\star_b \alpha_k, w} 
	= \scal{\alpha_1\otimes\cdots \otimes \alpha_k, \sum_{\pi_1,\ldots, \pi_k} 
	a_{\pi_1}\otimes\cdots \otimes a_{\pi_k}},
\end{equation*}
where the sum is over all disjoint subsets $\pi_1,\ldots, \pi_k$ of $[n]$ such that $\pi_1\cup\cdots\cup \pi_k = [n]$ and each $\pi_i$ is an interval (that is, the sum is over ordered Boolean partitions: Boolean partitions equipped with an arbitrary total ordering of their blocks).
\end{lem}

\begin{proof}
From the definition of $\Delta_b$, only Boolean partitions can occur in the expansion of $\scal{\alpha_1\star_b\cdots\star_b \alpha_k, w}$ that is obtained recursively from the associativity of $\star_b$. We get that this expansion has necessarily the form
$$
	\scal{\alpha_1\otimes\cdots \otimes \alpha_k, 
	\sum_{\pi_1,\ldots, \pi_k}\mu_\pi\cdot a_{\pi_1}\otimes\cdots \otimes a_{\pi_k}},
$$
for certain scalar coefficients $\mu_\pi$, where the sum is over ordered Boolean partitions $\pi=\{\pi_1,\ldots, \pi_k\}$. The statement in the lemma amounts therefore to the assertion that $\mu_\pi=1$, for any $\pi$.
However, as the $\star_b$ product is commutative (or, equivalently, as the coproduct $\Delta_b$ is cocommutative), $\mu_\pi=\mu_{\sigma{\pi}}$, for any permutation $\sigma$ of $[k]$ acting on $\pi$ by permutation of the blocks. We can therefore obtain a proof of the statement in the lemma by showing that $\mu_\pi=1$ in the particular case where the blocks of $\pi$ (which are intervals) are in natural order.

Let us prove this last property by induction on the number $k$ of blocks. The base case $k=1$ is obvious. Let us fix a canonically ordered Boolean partition $\pi$ on $[n]$ with $k$ blocks and assume that $\pi_k=\{p+1,\dots,n\}$.
The formula for the $\star_b$ product shows that, for an arbitrary linear form $\beta$ on $H_b$, we have that (as $\alpha_k$ is infinitesimal):
\begin{align*}
	\scal{\beta\star_b \alpha_k, w}
	&=\sum\limits_{r\leq n}\scal{\beta\otimes \alpha_k, a_1\cdots a_r\otimes a_{r+1}\cdots a_n
		+ a_{r+1}\cdots a_n \otimes  a_1\cdots a_r}\\
	& \quad +\sum\limits_{1\leq r<l< n}\scal{\beta\otimes \alpha_k, (a_1\cdots a_{r-1}|a_{l+1}\cdots a_n)\otimes a_{r}\cdots a_l}.
\end{align*}
Only the component $\scal{\beta\otimes \alpha_k, a_1\cdots a_p\otimes a_{p+1}\cdots a_n + a_{p+1}\cdots a_n \otimes a_1\dots a_p}$ is relevant for our computation. Replacing $\beta$ by $\alpha_1\star_b\cdots\star_b \alpha_{k-1}$, we get that $\mu_\pi$ is the coefficient of 
$$
	\scal{\alpha_1\otimes\cdots \otimes \alpha_{k-1},  a_{\pi_1}\otimes\cdots \otimes a_{\pi_{k-1}}},
$$
in the expansion of
$\scal{\alpha_1\star_b\cdots\star_b \alpha_{k-1}, a_1\cdots a_p}$, which is equal to 1, by the induction hypothesis. This completes the proof.
\end{proof}

\begin{thm}\label{thmfundBool}
Let $(\mathcal A,\varphi)$ be a non-commutative probability space and $\Phi$ the extension of $\varphi$ to $H_b$ as a character. Let $\alpha$ be the infinitesimal character 
$$
	\alpha:=\log^{\star_b}(\Phi).
$$
For the word $w=a_1 \cdots a_n \in T^+(\mathcal A)$ we obtain
\begin{equation*}
	\Phi(w)
		=\exp^{\star_b}(\alpha)(w) 
		= \sum_{j=1}^n \frac{\alpha^{\star_b j}}{j!}(w) 
		=\sum_{I \in \operatorname{Int}_n} \alpha_I(a_1, \ldots, a_n).
\end{equation*}
In particular, $igm(\alpha)={\mathrm{b}}$, the Boolean cumulant map.
\end{thm}

\begin{proof}
Indeed, by Lemma~\ref{lem:boolean_inf_char}:
\[
	\alpha^{\star_b k}(w) = k! \sum_{\pi}\prod_{\pi_i} \alpha(a_{\pi_i}),
\]
where the sum is over Boolean partitions with $k$ blocks $\pi_1,\ldots,\pi_k$. The factor $k!$ appears since there are $k!$ ways to order the blocks $\pi_1,\ldots, \pi_k$ to form an ordered Boolean partition. The conclusion follows by summation over $k$. 
\end{proof}

\begin{proof}[Proof of Proposition \ref{idprodb}]
We are now in the position to conclude the proof that the product in Proposition \ref{idprodb} identifies with the one in Definition \ref{defnBglawalg}. For notational clarity, we write temporarily $\ogreaterthan'$ for the product defined in Proposition \ref{idprodb}.

By Corollary \ref{coprolinB}, it is enough to prove that, given two elements $\phi$ and $\psi$ in $\mathcal G(\mathcal A)$, we have
$$
	\phi\ogreaterthan'\!\psi (K_{b,I})=\phi (K_{b,I})+\psi (K_{b,I})
$$
for $I$ arbitrary. For notational simplicity and without loss of generality, we treat the generic case where $I=(1,\ldots,n)$. 

By Theorem \ref{thmfundBool}, $\log^{\star_b}(\Phi)(b_1\cdots b_n)=\text{b}(b_1\cdots b_n)=\phi (K_{b,I})$. Therefore, as $\Delta_b$ encodes the definition of the product $\ogreaterthan'$,
$$
	\phi\ogreaterthan'\psi (K_{b,I})=\Phi\star_b\Psi (K_{b,I})
	=\log^{\star_b}(\Phi\star_b\Psi)(b_1\cdots b_n),
$$
where we write $\Phi$ and $\Psi$ for the extensions of $\phi$ respectively $\psi$ to characters on $H_b$.
However, as $\star_b$ is commutative, $\log^{\star_b}(\Phi\star_b\Psi)=\log^{\star_b}(\Phi)+\log^{\star_b}(\Psi)$, and we get finally
$$
	\phi\ogreaterthan'\psi (K_{b,I})=(\log^{\star_b}(\Phi)+\log^{\star_b}(\Psi))(b_1\cdots b_n)
	=\phi(K_{b,I})+\psi(K_{b,I}),
$$
which yields $\phi\ogreaterthan'\psi=\phi\ogreaterthan\psi$ and concludes the proof.
\end{proof}


\section{The free group of states}
\label{sect:fg}

We employ freely the notation used in the previous section. Although the arguments are essentially the same as in the Boolean case, we repeat them for the sake of completeness; the present article aims at being a reference on the topic of algebraic groups in non-commutative probability. The following definition and its motivations can be found in reference \cite{ebrahimipatras_19}. Recall the definition of the half-shuffle exponentials and logarithms in Section \ref{sect:GS}.

\begin{defn}[Free probability group of states, algebraic definition]\label{defnFreelawalg}
The set $\mathcal G(\mathcal A)$ equipped with the product
$$
	\phi\olessthan \psi:=gm(\mathcal{E}_\prec(\mathcal{L}_\prec(char(\phi))+\mathcal{L}_\prec(char(\psi)))
$$
is a group called the {\it free convolution group of states}.
\end{defn}

Recall also the moment-cumulant relations in free probability~\eqref{eq:freeMCrel}, which we rewrite as
\begin{equation}
\label{freec}
			\phi(a_1\cdots a_n)=\sum_{\pi \in NC_n} \text{k}_\pi(a_1 \cdots a_n),
\end{equation}
where as usual $\text{k}_\pi(a_1 \cdots a_n):=\prod_{\pi_k \in \pi} \text{k}_{|\pi_k|}(a_{\pi_k})$ and $NC_n$ stands for the set of noncrossing partitions of $[n]$. Given $(\mathcal A,\varphi)$, let $\phi \in T^+(\mathcal A)^\ast$ be the usual extension of the state map $\varphi$. Relations \eqref{freec} define implicitly the family of free cumulants $\{\text{k}_n\}_{n>0}$ in terms of the linear form $\phi$ respectively $\varphi$. By M\"obius inversion, we can express cumulants in terms of moments
\begin{equation}
\label{freec2}
			\mathrm{k}_n(a_1\cdots a_n)=\sum_{\pi \in NC_n}\mu_{\text{NC}}(\pi,1_n) \phi_\pi(a_1 \cdots a_n),
\end{equation}
where $\mu_{\text{NC}}(\pi,1_n)$ denotes the M\" obius function on the lattice of noncrossing partitions \cite{nicaspeicher_06}. Notice for later use that this equation shows that $\text{k}_n(a_1\cdots a_n)=\phi(a_1 \cdots a_n)+\text{l.o.t.}$, where $\text{l.o.t.}$ refers again to a sum of products of evaluations of $\varphi$ on words of length strictly less than $n$. Recall also from \cite{ebrahimipatras_17} the following statement:

\begin{prop}
With the same notation as above, we have
\begin{equation*}
	\phi =gm( \mathcal{E}_{\prec} (ichar({\mathrm{k}}))),
\end{equation*}
\begin{equation}
\label{def:freecumul2}
	\mathrm{k} =igm( \mathcal{L}_{\prec} (char(\phi))).
\end{equation}
\end{prop}

Let us translate these results into the language of algebraic groups.  Recall our notation: as before, we fix a basis  $\mathcal B=(b_i)_{i \in\NN}$ of $\mathcal A$ with $b_I:=b_{i_1} \cdots b_{i_k}$, for $I=(i_1,\ldots,i_k)\in{\mathbb N}^k$. The set $\mathcal Y$ of words $b_I$ defines then a (complete) system of coordinates on $\mathcal G(\mathcal A)$. 
The definition of the free probability group law on $\mathcal G(\mathcal A)$ implies that the terms in $\phi \olessthan \!\psi(b_I)$ rewrite as linear combinations of products of terms $\phi(b_J)$ and $\psi(b_K)$, where $J$ and $K$ stand for subsets of $I$. In other words, the $b_I$ are again representative functions on the group and $\mathcal G(\mathcal A)$ is the group of characters of its Hopf algebra of representative functions generated by the elements of $\mathcal Y$ (see again \cite{cp2021} for details). From \eqref{freec2} and the remark that follows the equation, the monomials
$$
	K_{f,I}:=\sum_{\pi \in NC_n}\mu_{\text{NC}}(\pi,1_n)\prod\limits_{\pi_i\in\pi}b_{\pi_i}
$$
form another system $\mathcal F$ of coordinates of the group that we call the {\it free coordinates}. The group law $\olessthan$ in Definition \ref{defnFreelawalg} and \eqref{def:freecumul2} yield
$$
	\mathcal{L}_{\prec}(char(\phi \olessthan \!\psi))=\mathcal{L}_{\prec}(char(\phi))  +\mathcal{L}_{\prec}(char(\psi))
$$
and
$$
	\phi \olessthan \!\psi (K_{f,I})=\phi(K_{f,I})+\psi(K_{f,I}).
$$
Hence, in the system of free coordinates, the group law is additive and we get the alternative presentation of the group $(\mathcal G(\mathcal A),\olessthan)$:

\begin{cor}\label{coprolinF}
The free probability group of states is, up to a canonical isomorphism, the group of characters of the Hopf algebra $S(\mathcal F)\cong S(T^+(\mathcal A))$ equipped with the product of polynomials over $\mathcal F$ and the coproduct 
$$
	\Delta_f(K_{f,I}):=K_{f,I} \otimes \mathbf 1 + \mathbf 1 \otimes K_{f,I}.
$$
\end{cor}

In the following proposition, we use again freely the notation of Section \ref{sect:up}. We write $\mathcal F_1$ and $\mathcal F_2$ for two copies of the tensor algebra $T(\mathcal A)$ and $\phi_1$, $\psi_2$ for two copies of elements $\phi, \psi \in \mathcal G(\mathcal A)$ acting respectively on $\mathcal F_1$ and $\mathcal F_2$. Given an element $a$ in $\mathcal A$, we write $a'$ and $a^{\prime\prime}$ for copies of $a$ in $\mathcal F_1$ respectively  $\mathcal F_2$.

\begin{prop}[Free probability group of states, probabilistic definition]
\label{idprodf}
The set $\mathcal G(\mathcal A)$ equipped with the product
$$
	\phi \olessthan \!\psi(a_1\cdots a_n):=\phi_1 \ast_f \psi_2((a_1'+a_1'')\cdots (a_n'+a_n'')), 
$$
is a group that identifies with the free probability group of states defined previously.
\end{prop}

We will decompose the proof that the two definitions of the product agree into several steps.
In the following arguments, we will use only the fact that the probabilistic definition of the product implies its associativity (the existence of a group structure will follow later from the identification of the two definitions of the product).
Let us first compute explicitly the product law of $\mathcal G(\mathcal A)$ as defined in Proposition \ref{idprodf} in the coordinates $\mathcal Y$, that is, the corresponding coproduct that we denote temporarily $\Delta_f$. We get from the definition of the product $\ast_f$ of states and Equation \eqref{adaptedfree}
the coproduct map:
$$
	\Delta_f  \colon T^+(\mathcal{A}) \to S(T^+(\mathcal{A})) \otimes S(T^+(\mathcal{A})),
$$ 
\begin{equation}
\label{coprodfree}
	\Delta_f(a_1 \cdots a_n) 
	:= \sum_{S \subseteq [n]}\sum\limits_{(\pi^1,\pi^2)\in ANC_n^S}\alpha_{\pi^1,\pi^2}\Big(\prod_{\pi_j\in\pi^1}a_{\pi_j}\Big)\otimes \Big(\prod_{\pi_k\in\pi^2}a_{\pi_k}\Big).
\end{equation} 
It is extended multiplicatively to the polynomial algebra $S(T^+(\mathcal{A}))$ by  $\Delta_f(\mathbf{1}):= \mathbf{1} \otimes \mathbf{1}$ and
$$
	\Delta_f(w_1 | \cdots | w_n) := \Delta_f(w_1) \cdots \Delta_f(w_n).
$$
As observed above, its coassociativity follows by duality from the associativity of $\olessthan$.
We can calculate the coproduct at low orders from the inductive formula for universal free products. Let $a_1a_2a_3a_4 \in T^+(\mathcal{A})$.
We have
\begin{align*}
	\Delta_f(a_1 a_2) 
		&= a_1 a_2 \otimes \mathbf 1 + a_1 \otimes a_2 + a_2 \otimes a_1 +  \mathbf 1 \otimes a_1 a_2.\\
	\Delta_f(a_1 a_2 a_3) 
		&= a_1a_2a_3 \otimes \mathbf{1} + a_1\otimes a_2a_3 + a_1a_2\otimes a_3 + a_1a_3\otimes a_2\\
		&\quad + a_2 \otimes a_1a_3 + a_3 \otimes a_1a_2 + a_2a_3 \otimes a_1 + \mathbf{1} \otimes a_1a_2a_3\\
	\Delta_f(a_1 a_2 a_3a_4) 
		&=(\id+\tau)( a_1a_2a_3a_4\otimes \mathbf{1} + a_1\otimes a_2a_3a_4 
								+ a_1a_2\otimes a_3a_4 + a_1a_4\otimes a_2a_3\\
		&\quad + a_1a_2a_3\otimes a_4 + a_1a_2a_4\otimes a_3 + a_1a_3a_4 \otimes a_2\\
		&\quad + (a_1|a_3) \otimes a_2a_4 + a_1a_3 \otimes (a_2|a_4) - (a_1|a_3) \otimes (a_2|a_4) ).
\end{align*}
Here, as before, the map $\tau$ switches tensor products, i.e., $\tau(x\otimes y):=y\otimes x.$

We write $H_f$ for the Hopf algebra $S(T^+(\mathcal{A}))$ equipped with the product of polynomials and with the cocommutative coproduct \eqref{coprodfree}. The corresponding convolution product of linear forms on $H_f$ is denoted $\star_f$. 

\begin{lem}\label{lem:K_{f,I}nf_char}
Let $\kappa_1,\ldots,\kappa_k$ be infinitesimal characters on $H_f$ and $w=a_1\cdots a_n \in T(\mathcal{A})$ a word. Then, we have:
\begin{equation*}
	\scal{\kappa_1\star_f \cdots \star_f \kappa_k, w} 
	= \scal{\kappa_1\otimes \cdots \otimes \kappa_k, \sum_{\pi_1,\ldots, \pi_k} a_{\pi_1}\otimes\cdots \otimes a_{\pi_k}},
\end{equation*}
where the sum is over all ordered noncrossing partitions of $[n]$ (that is noncrossing partitions equipped with an arbitrary total ordering of their blocks).
\end{lem}

\begin{proof}
From the definition of $\Delta_f$, only noncrossing partitions can occur in the expansion of
$\scal{\kappa_1\star_f \cdots \star_f \kappa_k, w}$ that is obtained recursively from the associativity of $\star_f$. This expansion has necessarily the form
$$
	\scal{\kappa_1\otimes\cdots \otimes \kappa_k, 
	\sum_{\pi_1,\ldots, \pi_k}\mu_\pi \cdot a_{\pi_1} \otimes \cdots \otimes a_{\pi_k}},
$$
for certain scalar coefficients $\mu_\pi$, where the sum is over ordered noncrossing partitions. The lemma amounts therefore to the statement that $\mu_\pi=1$, for any $\pi$. 
However, as $\Delta_f$ is cocommutative, the $\star_f$ product is commutative, i.e., $\mu_\pi=\mu_{\sigma{\pi}}$, for any permutation $\sigma$ of $[k]$ acting on $\pi$ by permutation of the blocks. We can therefore obtain a proof of the statement by showing that $\mu_\pi=1$ in the particular case where the blocks of $\pi$ are in a total order compatible with the usual tree ordering of blocks of noncrossing partition (that is, the one defined by $\pi_i>\pi_j$ if and only if $\forall k\in \pi_i,\exists l,m\in \pi_j$ such that $l<k<m$). This choice implies in particular that the maximal block $\pi_k$ for the total order is an interval $\pi_k=\{p+1,\dots,m\}$. We require furthermore that it is the rightmost interval of $\pi$. This last condition implies (by general elementary properties of noncrossing partitions) that $\pi-\{\pi_k\}$ cannot be decomposed as the union of a partition of $[p]$ and of $\{m+1,\dots n\}$.

Let us prove now the property by induction on $k$. The base case $k=1$ is obvious. Let us fix an ordered noncrossing partition $\pi$ of $[n]$ with $k$ blocks as above: in particular we assume that $\pi_k=\{p+1,\ldots,m\}$ is the rightmost interval in $\pi$.
The formula for the $\star_f$ product shows that, for an arbitrary linear form $\beta$ on $H_f$, we have that (as $\kappa_k$ is infinitesimal):
\begin{align*}
	\scal{\beta\star_f \kappa_k, w}
	&=\sum\limits_{1\leq r<l\leq n}\scal{\beta\otimes \kappa_k, a_1\cdots a_{r-1}a_{l+1}\cdots a_n\otimes a_{r}\cdots a_l}\\
	&\quad +\sum\limits_{1<r<l< n}\scal{\beta\otimes \kappa_k, \alpha_{\{\{1,\ldots,r-1\},\{l+1,\ldots,n\}\},\{\{r,\ldots,l\}\}}(a_1\cdots a_{r-1}|a_{l+1}\cdots a_n)\otimes a_{r}\cdots a_l}.
\end{align*}	
As $\pi-\{\pi_k\}$ cannot be decomposed as the union of a partition of $[p]$ and of $\{m+1,\ldots, n\}$, only the component $\scal{\beta\otimes \kappa_k, a_1\cdots a_pa_{m+1} \cdots a_n \otimes a_{p+1}\cdots a_m}$ is relevant for our computation. Replacing $\beta$ by $\kappa_1\star_f \cdots \star_f \kappa_{k-1}$, we get that $\mu_\pi$ is the coefficient of
$$
	\scal{\kappa_1\otimes\cdots \otimes \kappa_{k-1},  a_{\pi_1} \otimes\cdots\otimes a_{\pi_{k-1}}},
$$
in the expansion of
$\scal{\kappa_1\star_f \cdots \star_f \kappa_{k-1}, a_1\cdots a_pa_{m+1}\cdots a_n}$, that is equal to one, by the induction hypothesis. This completes the proof.
\end{proof}

\begin{thm}\label{thmfundfree}
Let $(\mathcal A,\varphi)$ be a non-commutative probability space and $\Phi$ the extension of $\varphi$ to $H_f$ as a character. Let $\kappa$ be the infinitesimal character 
$$
	\kappa:=\log^{\star_f}(\Phi).
$$
For the word $w=a_1 \cdots a_n \in T^+(\mathcal A)$ we get:
\begin{equation*}
	\Phi(w)
		=\exp^{\star_f}(\kappa)(w) 
		= \sum_{j=1}^n \frac{\kappa^{\star_f j}}{j!}(w) 
		=\sum_{\pi \in NC_n} \kappa_\pi(a_1, \ldots, a_n).
\end{equation*}
In particular, $igm(\kappa)={\mathrm{k}}$, the free cumulant map.
\end{thm}

\begin{proof}
Indeed, by Lemma~\ref{lem:K_{f,I}nf_char}:
\[
	\kappa^{\star_f k}(w) = k! \sum_{\pi}\prod_{\pi_i} \kappa(a_{\pi_i}),
\]
where the sum is over all $\pi \in NC_n$ with $k$ blocks $\pi_1,\ldots,\pi_k$. The factor $k!$ appears since there are $k!$ ways to order the blocks. The conclusion follows by summation over $k$. 
\end{proof}

\begin{proof}[Proof of Proposition \ref{idprodf}]
We are now in the position to conclude the proof that the product in Proposition \ref{idprodf} identifies with the one in Definition \ref{defnFreelawalg}.  For notational clarity, we write temporarily $\olessthan'$ for the product defined in Proposition \ref{idprodf}.

By Corollary \ref{coprolinF}, it is enough to prove that, given two elements $\phi$ and $\psi$ in $\mathcal G(\mathcal A)$, 
$$
	\phi\olessthan'\!\psi (K_{f,I})=\phi (K_{f,I})+\psi (K_{f,I}),
$$
for $I$ arbitrary. For notational simplicity and without loss of generality, we treat the generic case where $I=(1,\dots,n)$. 

By Theorem \ref{thmfundfree}, we have $\log^{\star_f}(\Phi)(b_1\cdots b_n)=\text{k}(b_1\cdots b_n)=\phi (K_{f,I})$. Therefore, as $\Delta_f$ encodes the definition of the product $\olessthan'$,
$$
	\phi\olessthan'\psi (K_{f,I})=\Phi\star_f\Psi (K_{f,I})
	=\log^{\star_f}(\Phi\star_f\Psi)(b_1\cdots b_n),
$$
where we write $\Phi$ and $\Psi$ for the extension of $\phi$ and $\psi$ to characters on $H_f$.
However, as $\star_f$ is commutative, $\log^{\star_f}(\Phi\star_f\Psi)=\log^{\star_f}(\Phi)+\log^{\star_f}(\Psi)$, and we get finally
$$
	\phi\olessthan'\psi (K_{f,I})=(\log^{\star_f}(\Phi)+\log^{\star_f}(\Psi))(b_1\cdots b_n)
	=\phi(K_{f,I})+\psi(K_{f,I}),
$$
which yields $\phi\olessthan'\psi=\phi\olessthan\psi$ and concludes the proof.
\end{proof}

\medskip
\medskip


\end{document}